\documentclass[a4paper,twoside,reqno]{amsart}

\usepackage{graphicx}
\usepackage{a4wide}
\usepackage{amstext, amsmath, amssymb, amsthm}
\usepackage{latexsym}
\usepackage{booktabs}
\usepackage{url}
\usepackage{import}
\usepackage{fancyvrb}
\usepackage{stmaryrd}
\usepackage{pdfsync}
\usepackage{algorithm}
\usepackage{tikz}
\usepackage{pgfplots}
\usepackage{pgfplotstable}
\usepackage{tabularx}
\usepackage{todonotes}
\usepackage[numbers,square,sort&compress]{natbib}
\usepackage{caption}
\usepackage{subcaption}
\usepackage{enumitem}

\DeclareMathOperator{\Div}{div}

\DeclareMathOperator{\dist}{dist}

\newcommand{\RR}{\mathbb{R}}

\newcommand{\VV}{\mathbb{V}}

\newcommand{\mcQ}{\mathcal{Q}}
\newcommand{\mcV}{\mathcal{V}}
\newcommand{\mcK}{\mathcal{K}}

\newcommand{\mcW}{\mathcal{W}}
\newcommand{\mcT}{\mathcal{T}}
\newcommand{\mcH}{\mathcal{H}}

\newcommand{\tn}{|\mspace{-1mu}|\mspace{-1mu}|}

\newcommand{\Gammah}{{\Gamma_h}}
\newcommand{\nablas}{\nabla_\Gamma}
\newcommand{\nablash}{\nabla_{\Gamma_h}}
\newcommand{\divs}{\Div_{\Gamma}}

\newcommand{\Ps}{{P_{\Gamma}}}
\newcommand{\Qs}{{Q_{\Gamma}}}
\newcommand{\Psh}{{P}_{\Gamma_h}}
\newcommand{\Qsh}{{Q}_{\Gamma_h}}

\newcommand{\foralls}{\forall\,}

\newcommand{\dr}{\,\mathrm{d}r}

\newcommand{\ds}{\,\mathrm{d} \Gamma}
\newcommand{\dsh}{\,\mathrm{d} \Gamma_h}

\newcommand{\onehalf}{\frac{1}{2}}

\usepackage[plainpages=false]{hyperref}
\hypersetup{
  bookmarksopen=true,
                bookmarksnumbered=true,
                pdfborder={0 0 0},
                pdftitle= {Stabilized Cut Finite Element Methods for the Darcy Problem on Surfaces},
                pdfauthor={P. Hansbo, M.G. Larson, A. Massing},
                pdfcreator={P. Hansbo, M.G. Larson, A. Massing},
                colorlinks=true,
                linkcolor=blue,
                urlcolor=blue
        }

\usepackage{verbatim} 
\DefineVerbatimEnvironment{code}{Verbatim}{frame=single,rulecolor=\color{blue}}

        \newtheorem{theorem}{Theorem}[section]

\newtheorem{lemma}[theorem]{Lemma}
\newtheorem{corollary}[theorem]{Corollary}

\theoremstyle{remark}
\newtheorem{remark}[theorem]{Remark}

\theoremstyle{definition}

\numberwithin{equation}{section}
\numberwithin{figure}{section}
\numberwithin{table}{section}

\begin{document}
\title[A Stabilized CutFEM for the Darcy Surface Problem]{\bf
  A Stabilized Cut Finite Element Method for the Darcy Problem on Surfaces}


\author[P. Hanbso]{Peter Hansbo}
\address[Peter Hansbo]{Department of Mechanical Engineering, J\"onk\"oping University,
SE-55111 J\"onk\"oping, Sweden.}
\email{peter.hansbo@ju.se}

\author[M. G. Larson]{Mats G.\ Larson}
\author[A. Massing]{Andr\'e Massing}
\address[Mats G.\ Larson, Andr\'e Massing]{Department of Mathematics and Mathematical Statistics, Ume{\aa} University, SE-90187 Ume{\aa}, Sweden}
\email{mats.larson@umu.se}
\email{andre.massing@umu.se}

\date{\today}


\maketitle

\begin{abstract}
{We develop a cut finite element method for the Darcy problem on
  surfaces.  The cut finite element method is based on embedding the
  surface in a three dimensional finite element mesh and using finite
  element spaces defined on the three dimensional mesh as trial and test
  functions.
  Since we consider a partial differential equation on a
  surface, the resulting discrete weak problem might be severely ill conditioned.
  We propose a full gradient and a normal gradient based stabilization
  computed on the background mesh to render the proposed
  formulation stable and well conditioned irrespective of the surface
  positioning within the mesh.
  Our formulation extends and simplifies the Masud-Hughes stabilized
  primal mixed formulation of the Darcy surface problem proposed
  in~\cite{HansboLarson2016} on fitted triangulated surfaces.
  The tangential condition on the velocity and the pressure
  gradient is enforced only weakly, avoiding the need for any
  tangential projection.
  The presented numerical analysis accounts for different polynomial orders for the
  velocity, pressure, and geometry approximation which are corroborated
  by numerical experiments.
  In particular, we demonstrate both theoretically and through numerical results
  that the normal gradient stabilized variant results in a high order scheme.
  }
\end{abstract}

\section{Introduction}
\label{sec:introduction}

\subsection{Background and Earlier Work}
In recent years, there has been a rapid development of cut finite
element methods, also called trace finite element methods, for the
numerical solution of partial differential equations (PDEs) on complicated or
evolving surfaces embedded into $\RR^d$.
The main idea
is to use the restriction of finite element basis
functions defined on a $d$-dimensional background mesh to a discrete,
piecewise smooth surface representation which is allowed to cut
through the mesh in an arbitrary fashion.
The active background mesh then consists of all elements which are cut
by the discrete surface, and the finite element space restricted to
the active mesh is used to discretize the surface PDE.
This approach was first proposed in \cite{OlshanskiiReuskenGrande2009}
for the Laplace-Beltrami on a closed surface,
see also~\cite{BurmanClausHansboEtAl2014} and the references therein
for an overview of cut finite element techniques.

Depending on the positioning of the discrete surface
within the background mesh,
the resulting system matrix might be severely ill conditioned
and either preconditioning 
\cite{OlshanskiiReusken2014} or 
stabilization \cite{BurmanHansboLarson2015} is 
necessary to obtain a well conditioned linear system.
The stabilization introduced and analyzed
in~\cite{BurmanHansboLarson2015} is based on so called face
stabilization or ghost penalty, which provides control over the
jump in the normal gradient across interior faces in the active mesh.
In particular, it was shown that the condition number scaled in an
optimal way, independent of how the surface cut the background mesh.
Thanks its versatility, the face based stabilization can naturally be
combined with discontinuous cut finite element methods as demonstrated
in~\cite{BurmanHansboLarsonEtAl2016a}.
To reduce the matrix stencil and ease the implementation,
a particular simple low order, full gradient stabilization
using continuous piecewise linears
was developed and
analyzed in~\cite{BurmanHansboLarsonEtAl2016c}
for the Laplace-Beltrami operator.
Then a unifying abstract framework for analysis of cut finite element
methods on embedded manifolds of arbitrary codimension was developed
in~\cite{BurmanHansboLarsonEtAl2016} and, in particular, the normal
gradient stabilization term was introduced and analyzed.
Further developments include convection
problems~\cite{BurmanHansboLarsonEtAl2015b, OlshanskiiReuskenXu2014},
coupled bulk-surface problems~\cite{BurmanHansboLarsonEtAl2014, GrossOlshanskiiReusken2014}
and higher order versions of trace fem for the Laplace-Beltrami
operator~\cite{Reusken2014,GrandeLehrenfeldReusken2016}.
Moreover, extensions to time-dependent, parabolic-type problems on evolving
domains were proposed in~\cite{HansboLarsonZahedi2016, OlshanskiiReuskenXu2014a}.

So far, with their many applications to fluid dynamics, material science
and biology, e.g.,
\cite{GanesanTobiska2009, GrossReusken2011, NovakGaoChoiEtAl2007, EilksElliott2008, ElliottStinnerVenkataraman2012, BarreiraElliottMadzvamuse2011},
mainly scalar-valued,
second order elliptic and parabolic type equations 
have been considered in the context of
cut finite element methods for surface PDEs.
Only very recently, vector-valued surface PDEs
in combination with unfitted finite element technologies
have been considered, for instance in the numerical discretization
of surface-bulk problems modeling flow dynamics
in fractured porous
media~\cite{FumagalliScotti2013,
DelPraFumagalliScotti2015,
AntoniettiFacciolaRussoEtAl2016,
FlemischFumagalliScotti2016}.
But while these contributions employed cut finite element type methods
to discretize the bulk equations, only fitted (mixed and stabilized)
finite elements methods on triangulated surfaces have been developed for
vector surface equation such as the Darcy surface problem, see for
instance~\cite{FerroniFormaggiaFumagalli2016,HansboLarson2016}.
The present contribution is the first where a cut finite element
method for a system of partial differential equations on a surfaces
involving tangent vector fields of partial differential equations is
developed.

\subsection{New Contributions}
We develop a stabilized cut finite element method 
for the numerical solution of the Darcy problem on a surface.
The proposed variational formulation
follows the approach in~\cite{HansboLarson2016} 
for the Darcy problem on 
triangulated surfaces which is based on the stabilized
primal mixed formulation by~\citet{MasudHughes2002}.
Note that standard mixed type elements are typically not
available on discrete cut surfaces.
Combining the ideas from~\cite{HansboLarson2016} with the stabilized full
gradient formulations of the Laplace-Beltrami problem
from~\cite{BurmanHansboLarsonEtAl2016c,BurmanHansboLarsonEtAl2016},
the tangent condition on both the velocity and the pressure gradient is
enforced only weakly.
When employing finite element function from the full $d$-dimensional
background mesh, such a weak enforcement of the tangential condition
is convenient and rather natural.

To render the proposed formulation stable and well conditioned
irrespective of the relative surface position in the background mesh,
we consider two stabilization forms: the full gradient stabilization introduced in 
\cite{BurmanHansboLarsonEtAl2016c} which is convenient for 
low order elements, and the normal gradient stabilization 
introduced in \cite{BurmanHansboLarsonEtAl2016} which also 
works for higher order elements.
Through these stabilizations, we gain control of the variation of the
solution orthogonal to the surface, which in combination with the
Masud-Hughes variational formulation results in a coercive formulation
of the Darcy surface problem.
In practice, the exact 
surface is approximated leading to a geometric error which 
we take into account in the error analysis. We show stability 
of the method and establish optimal order a priori error estimates.
The presented numerical analysis also accounts for different polynomial orders for the
velocity, pressure, and geometry approximation.

\subsection{Outline}
The paper is organized as follows: In Section~\ref{sec:darcy-problem-cont} we present 
the Darcy problem on a surface together with its Masud-Hughes weak 
formulation, followed by the formulation of the cut finite element 
method in Section~\ref{sec:darcy-problem-cutfem}.
In Section~\ref{sec:preliminaries} we collect a number of preliminary theoretical 
results, which will be needed in Section~\ref{sec:a-priori-est}
where the main
a priori error estimates in the 
energy and $L^2$ norm are established.
Finally, in Section~\ref{sec:numerical-results} we present numerical results
illustrating the theoretical findings of this work.

\section{The Darcy Problem on a Surface}
\label{sec:darcy-problem-cont}

\subsection{The Continuous Surface}
\label{ssec:preliminaries}
In what follows, $\Gamma$ denotes a smooth compact hypersurface
without boundary which is embedded in ${{\RR}}^{k}$ and equipped with a
normal field $n: \Gamma \to \RR^{d}$ and signed distance
function $\rho$.
Defining
the tubular neighborhood of $\Gamma$ by
$U_{\delta_0}(\Gamma) = \{ x \in \RR^{d} : \dist(x,\Gamma) < \delta_0
\}$, the closest point projection $p(x)$ 
is the uniquely defined mapping given by
\begin{align}
  p(x) =  x - \rho(x) n(p(x))
\label{eq:closest-point-projection}
\end{align}
which maps $x \in U_{\delta_0}(\Gamma)$ to the unique point $p(x) \in
\Gamma$ such that $| p(x) - x | = \dist(x, \Gamma)$ for some $\delta_0 > 0$, see~\cite{GilbargTrudinger2001}.
The closest point projection allows the extension of a function $u$ defined on $\Gamma$ to its
tubular neighborhood $U_{\delta_0}(\Gamma)$ using the pull back
\begin{equation}
\label{eq:extension}
u^e(x) = u \circ p (x)
\end{equation}
In particular, we can smoothly extend the normal field $n_{\Gamma}$ to the tubular neighborhood $U_{\delta_0}(\Gamma)$.
On the other hand, 
for any subset $\widetilde{\Gamma} \subseteq U_{\delta_0}(\Gamma)$ such that 
$p: \widetilde{\Gamma} \to \Gamma $ is bijective, a function $w$ on
$\widetilde{\Gamma}$ can be lifted to $\Gamma$ by the push forward satisfying
\begin{align}
  (w^l)^e = w^l \circ p = w \quad \text{on } \widetilde{\Gamma}
\end{align}
Finally, for any function space $V$ defined on $\Gamma$, we denote
the space consisting of extended functions by $V^e$ and
correspondingly, the notation $V^l$ refers to the lift of a
function space~$V$ defined on $\widetilde{\Gamma}$. {{}

\subsection{The Surface Darcy Problem}
To formulate the Darcy problem on a surface, we first recall
the definitions of the surface gradient and divergence.
For a function $p: \Gamma \rightarrow \RR$
the tangential gradient of $p$ can be expressed as
\begin{equation}
  \nablas p = \Ps \nabla p^e,
  \label{eq:tangent-gradient}
\end{equation}
where $\nabla$ is the standard ${{\RR}}^{d}$ gradient and $\Ps = \Ps(x)$ denotes the
orthogonal projection of $\RR^{d}$ onto the tangent plane
$T_x\Gamma$ of $\Gamma$ at $x \in \Gamma$
given by
\begin{equation}
  \Ps = I - n \otimes n,
\end{equation}
where $I$ is the identity matrix.
Since $p^e$ is constant in the normal direction, 
we have the identity 
\begin{equation}\label{eq:grad-full}
\nabla p^e = \Ps \nabla p^e \quad \text{on $\Gamma$}.
\end{equation}
Next, the surface divergence of a vector field $u:\Gamma \rightarrow \RR^d$ 
is defined by
\begin{equation}
\divs(u) = \text{tr}(u \otimes \nablas ) 
= \text{div}(u^e) - n \cdot (u^e \otimes \nabla)\cdot n.
\end{equation}
With these definitions, the Darcy problem on the surface $\Gamma$
is to seek the tangential velocity vector field
$u_t:\Gamma \rightarrow T(\Gamma)$ 
and the pressure $p:\Gamma\rightarrow \RR$ 
such that
\begin{subequations}
  \label{eq:darcy-problem-cont}
\begin{alignat}{2}\label{eq:Darcya}
\divs u_t &= f \qquad &\text{on $\Gamma$},
\\ \label{eq:Darcyb}
u_t + \nablas p &= g \qquad &\text{on $\Gamma$}.
\end{alignat}  
\end{subequations}
Here, $f:\Gamma \rightarrow \RR$ is a given function 
such that $\int_\Gamma f = 0$ and $g:\Gamma \rightarrow \RR^d$
is a given tangential vector field. 

\subsection{The Masud-Hughes Weak Formulation}
We follow \cite{HansboLarson2016} and base our finite element method 
on an extension of the Masud-Hughes weak formulation, originally 
proposed in \cite{MasudHughes2002} for planar domains, to the surface 
Darcy problem.
Using Green's formula 
\begin{equation}\label{eq:Greensdiv}
(\divs v_t, q)_\Gamma = -(v_t, \nablas q)_\Gamma
\end{equation}
valid for tangential vector fields $v_t$, a direct application of
the original Masud-Hughes formulation is built upon the fact
that the Darcy problem~(\ref{eq:darcy-problem-cont}) solves the weak problem
\begin{equation}
  \label{eq:darcy-problem-weak-tang}
  a((u_t, p), (v_t, q)) = l((v_t, q))
\end{equation}
for test functions $v \in [L^2(\Gamma)]$ and $q \in H^1(\Gamma)/\RR$ with
\begin{align}
  \label{eq:darcy-problem-weak-tang-a}
  a((u_t, p), (v_t, q))
  &=
  (u_t, v_t)_{\Gamma} + (\nablas p, v_t)_{\Gamma} - (u_t, \nablas q)_{\Gamma}
  + \onehalf (u_t + \nablas p, - v_t + \nablas q)_{\Gamma}
  \\
  \label{eq:darcy-problem-weak-tang-l}
  l((v,q))
  &= (f, q)_{\Gamma} + (g, v_t)_{\Gamma} + \onehalf (g, - v_t + \nablas q)_{\Gamma}
\end{align}
As in \cite{HansboLarson2016} we enforce the tangent condition 
on the velocity weakly by using full velocity fields
in formulation~\eqref{eq:darcy-problem-weak-tang}  
and not only their tangential projection.
But in contrast to~\cite{HansboLarson2016} we also employ the identity
(\ref{eq:grad-full}) to replace the pressure related tangent gradients with the full
gradient in order to simplify the implementation further.
Earlier, such full gradient based formulation
have been developed for the Poisson problem on the surface,
see~\cite{Reusken2014,BurmanHansboLarsonEtAl2016c}.
With $\mcV = [L^2(\Gamma)]^3$ as the velocity space, $\mcQ=H^1(\Gamma)/\RR$ as the pressure space
and $\VV = \mcV \times \mcQ$ as the total space,
the resulting Masud-Hughes weak formulation of the Darcy surface problem~(\ref{eq:darcy-problem-cont})
is to seek $U := (u,p) \in \VV$ satisfying
\begin{equation}
  \label{eq:mhweak}
A(U, V) = L(V) \quad 
\end{equation}
for all $V := (v,q) \in \VV$, where 
\begin{align}
  A(U,V) 
  &=(u,v)_\Gamma 
  +(\nabla p,v)_{\Gamma}  
  -(u, \nabla q)_\Gamma 
  +\frac{1}{2}(u + \nabla p,-v + \nabla q)_\Gamma,
  \label{eq:a}
  \\
  L(V) &= (f,q)_\Gamma + (g,v)_\Gamma
  +\frac{1}{2}(g,-v + \nabla q)_\Gamma.
  \label{eq:l}
\end{align} 
Expanding the forms, the bilinear form $A$ and linear form $L$ can be rewritten as
\begin{align}
A(U,V)  
&= 
\frac{1}{2}(u,v)_\Gamma 
+ \frac{1}{2}(\nabla p,\nabla q)_\Gamma
+\frac{1}{2}(\nabla p,v)_{\Gamma}  
-\frac{1}{2}(u, \nabla q)_\Gamma,
\\
L(V) 
&= 
(f,q)_\Gamma  
+ \frac{1}{2}(g,v + \nabla q)_\Gamma
\end{align} 
and consequently, the bilinear form $A$ consists of a symmetric positive 
definite part and a skew symmetric part.
For a more detailed discussion on various weak formulation of the surface Darcy problem,
we refer to~\cite{HansboLarson2016}.
Finally, note that since $\Gamma$ is smooth and $p\in Q$ is the solution to the 
elliptic problem $\divs(\nablas p) = \divs u_t - \divs g 
= f  - \divs g$, the following elliptic regularity 
estimate holds
\begin{equation}\label{eq:ellreg}
\| u_t \|_{H^{s+1}(\Gamma)} 
+
\| p \|_{H^{s+2}(\Gamma)} 
\lesssim \|f \|_{H^{s}(\Gamma)}
+
\| g \|_{H^{s+1}(\Gamma)}
\end{equation}

\section{Cut Finite Element Methods for the Surface Darcy Problem}
\label{sec:darcy-problem-cutfem}

\subsection{The Discrete Surface and Active Background Mesh}
\label{ssec:domain-disc}
For $\Gamma$, we assume that the discrete surface approximation $\Gammah$ is
represented by a piecewise smooth surface
consisting of smooth $d-1$ dimensional surface parts $\mcK_h = \{ K\} $
associated with a piecewise smooth normal field $n_h$.
On $\Gamma_h = \bigcup_{K \in \mcK_h} K$ we can then define
the discrete tangential projection $\Psh$
as the pointwise orthogonal projection on the $d$-dimensional tangent space defined
for each $x \in K$ and $K\in \mcK_h$.
We assume that:
\begin{itemize}
  \item $\Gamma_h \subset U_{\delta_0}(\Gamma)$ and that the closest
    point mapping $p:\Gamma_h \rightarrow \Gamma$ is a bijection for
    $0 < h \leq h_0$.
  \item The following estimates hold
    \begin{equation} \
      \| \rho \|_{L^\infty(\Gamma_h)} \lesssim h^{k_g+1}, \qquad
      \| n^e - n_h \|_{L^\infty(\Gamma_h)} \lesssim h^{k_g}.
      \label{eq:geometric-assumptions-II}
    \end{equation}
  \end{itemize}
Let $\widetilde{\mcT}_{h}$ be a quasi-uniform mesh, with mesh parameter
$0<h\leq h_0$, which consists of shape regular closed simplices
and covers some open and bounded domain $\Omega \subset \RR^{k}$ containing
the embedding neighborhood $U_{\delta_0}(\Gamma$).
For the background mesh $\widetilde{\mcT}_{h}$ we define the \emph{active} (background)
$\mcT_h$ mesh 
\begin{align}
  \mcT_h &= \{ T \in \widetilde{\mcT}_{h} : T \cap \Gamma_h \neq \emptyset \},
  \label{eq:narrow-band-mesh}
\end{align}
see~Figure~\ref{fig:domain-set-up} for an illustration. Finally,
for the domain covered by $\mcT_h$ we introduce the notation
\begin{align}
  N_h &= \cup_{T \in \mcT_h} T.
  \label{eq:Nh-def}
\end{align}
\begin{figure}[htb]
  \begin{center}
    \includegraphics[width=0.45\textwidth]{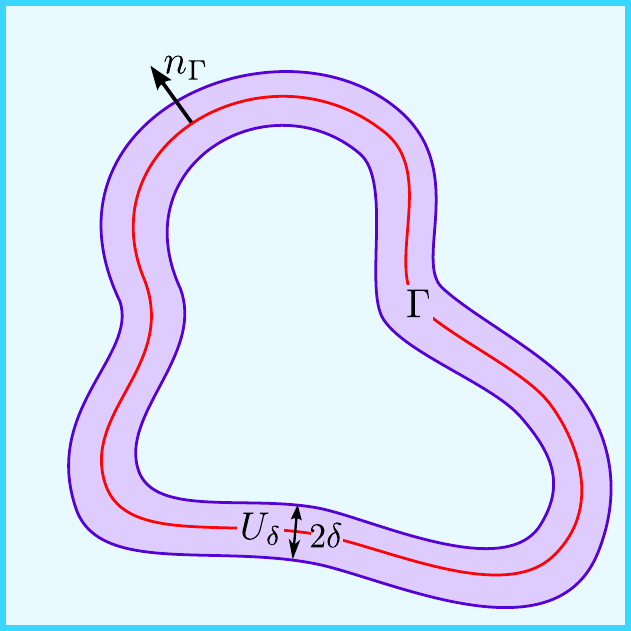}
    \hspace{0.03\textwidth}
    \includegraphics[width=0.45\textwidth]{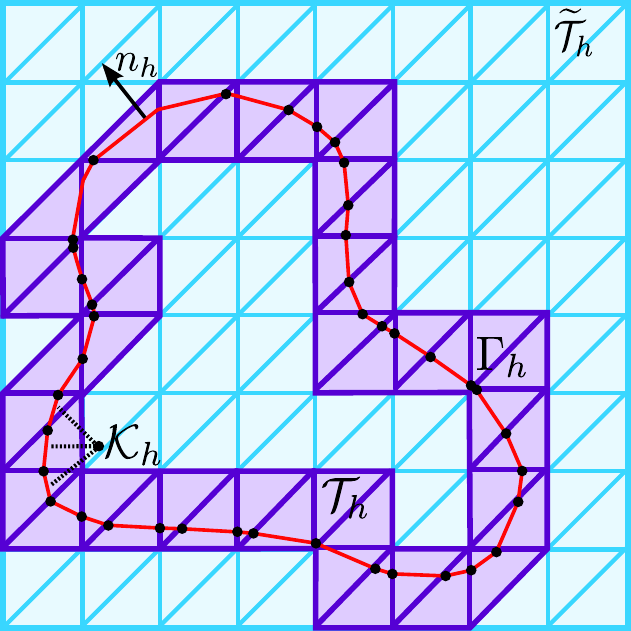}
  \end{center}
  \caption{Set-up of the continuous and discrete domains. (Left) Continuous surface $\Gamma$
  enclosed by a $\delta$ tubular neighborhood $U_{\delta}(\Gamma)$.
  (Right) Discrete manifold $\Gamma_h$ embedded into a background mesh
  $\widetilde{\mcT}_h$ from which the active (background) mesh $\mcT_h$ is extracted.
  }
  \label{fig:domain-set-up}
\end{figure}

\subsection{Stabilized Cut Finite Element Methods}
On the active mesh $\mcT_h$ we introduce
the discrete space of continuous piecewise polynomials of order $k$,
\begin{equation}
  X_h^k = \{ v \in C(N_h) : v|_T \in  P_k(T) \; \foralls T \in \mcT_h \}.
  \label{eq:Xh-def}
\end{equation}
Occasionally, if the order is not of particular importance, we simply drop the superscript $k$.
Next, the discrete velocity, pressure and total approximations spaces are defined by
respectively
\begin{align}
  \mcV_h = [X_h^{k_u}]^d,
  \qquad
  \mcQ_{h} = \{ v \in X_h^{k_p} : \lambda_{\Gamma_h}(v) = 0 \},
  \label{eq:Qh-def}
  \qquad
  \mcW_h = \mcV_h \times \mcQ_h,
\end{align}
where we explicitly permit different approximation orders $k_u$ and $k_p$
for the velocity and pressure space.
As in the continuous case, $\lambda_{\Gamma_h}(\cdot)$
denotes the average operator on $\Gammah$ defined by $\lambda_{\Gamma_h}(v) =
\tfrac{1}{\Gammah}\int_{\Gamma_h} v$.
Now the stabilized, full gradient cut finite element method for the
surface Darcy problem~(\ref{eq:darcy-problem-cont}) is to
seek $U_h := (u_h, p_h)  \in  \VV_h$ such that for all $V := (v, p) \in \VV_h$,
\begin{align}
  B_h(U_h, V) = L_h(V) 
  \label{eq:darcy-probl-cutfem}
\end{align}
where
\begin{align}
  B_h(U_h, V) &= A_h(U_h, V) + S_h(U_h, V) = L_h(V),
  \\
  \label{eq:Ah-def}
  A_h(U_h, V) &=  (u_h,v)_{\Gammah}
                    + (\nabla p, v)_{\Gammah}
                    - (u_h,\nabla q)_{\Gammah}
                    + \onehalf ( u + \nabla p, -v + \nabla
                    q)_{\Gammah},
 \\
  S_h(U_h, V) &= s_h(u_h,v_h) + s_h(p_h,q_h),
  \label{eq:Sh-def}
  \\
  L_h(V) &= (f,q) + \onehalf (g, v + \nabla q)_{\Gamma}.
  \label{eq:Lh-def}
\end{align}
For the stabilization form $s_h$, two realizations will be proposed in this work.
First, we consider a full gradient based stabilization originally introduced for
Laplace-Beltrami surface problem in~\cite{BurmanHansboLarsonEtAl2016c},
\begin{alignat}{3}
  s_h^1(u_h, v_h) &= h (\nabla u_h, \nabla v_h)_{\mcT_h},
                    \qquad
  & & s_h^1(p_h, q_h) = h (\nabla p_h, \nabla q_h)_{\mcT_h}.
  \\
  \intertext{Second, to devise a higher order approximation scheme,
  the normal gradient stabilization}
  s_h^2(u_h, v_h) &= h (n_h \cdot \nabla u_h, n_h \cdot \nabla v_h)_{\mcT_h},
                    \qquad
  & &s_h^2(p_h, q_h) = h (n_h \cdot \nabla p_h, n_h \cdot \nabla q_h)_{\mcT_h}
\end{alignat}
first proposed and analyzed in~\cite{BurmanHansboLarsonEtAl2016}
and then later also considered in~\cite{GrandeLehrenfeldReusken2016},
will be employed. In the remaining work, we will simply write $S_h$ and $s_h$
without superscripts as long as no distinction between the stabilization forms is needed.
\begin{remark}
  For the normal gradient stabilization,
  any $h$-scaling of the form $h^{\alpha - 1}$ with $\alpha \in [0,2]$ gives
  an eglible stabilization operator, as pointed out in~\cite{BurmanHansboLarsonEtAl2016}.
  The condition $\alpha \leqslant 2$ guarantees that the stabilization
  result~\ref{lem:scaled-L2-norm-Nh-stab}
  for a properly scaled $L^2$ norm holds, the condition $\alpha
  \geqslant 0$ on the other hand assures that the condition number
  of the discrete linear system scales with the mesh size similar to
  the triangulated surface
  case. We refer to~\cite{BurmanHansboLarsonEtAl2016} for further details.
\end{remark}

\section{Preliminaries}
\label{sec:preliminaries}
To prepare the forthcoming analysis of the proposed cut finite element method
in Section~\ref{sec:a-priori-est},
we here collect and state a number of useful definitions, approximation results
and estimates.
In particular, we introduce suitable continuous and discrete norms,
review the construction of a proper interpolation operator and
recall the fundamental geometric estimates needed to quantify the
quadrature errors introduced by the discretization of $\Gamma$.

\subsection{Discrete Norms and Poincar\'e Inequalities}
\label{ssec:discrete-norms-poincare}
The symmetric parts of the bilinear forms $A$ and $A_h$
give raise to the following natural continuous and discrete ``energy''-type norms
\begin{align}
  \tn  U \tn^2 = \| u \|_{\Gamma}^2 + \| \nabla p \|_{\Gamma}^2,
  \qquad
  \tn  U_h \tn_h^2 = \| u_h \|_{\Gammah}^2 + \| \nabla p_h \|_{\Gammah}^2 +
  | U_h |_{S_h},
  \label{eq:triple-norm-def}
\end{align}
where $|\cdot|_{S_h}$ denotes the semi-norm induced by the
stabilization form $S_h$, 
\begin{align}
  | U_h |_{S_h}^2 
  = S_h(U_h, U_h).
  \label{eq:sh-norm-def}
\end{align}
To show that $\tn \cdot \tn_h$ actually defines a proper norm,
we recall the following result from~\cite{BurmanHansboLarsonEtAl2016}.
\begin{lemma}
  \label{lem:scaled-L2-norm-Nh-stab} For $v \in X_h$, the following estimate holds
  \begin{align}
    h^{-1}\| v \|^2_{\mcT_h}
    & \lesssim
      \| v \|_{\Gamma_h}^2
      + s_h^i(v,v)
      \quad \text{for } i = 1,2,
    \label{eq:discrete-poincare-Nh-normal-grad}
  \end{align}
  for $0<h \leq h_0$ with $h_0$ small enough.
\end{lemma}
\begin{remark}
  Simple counter examples show that the sole expression $\| v_h
\|_{\Gammah} + \| \nabla q_h \|_{\Gammah}$ defines only a semi-norm on
$\mcV_h \times \mcQ_h$. For instance, let $ \Gamma = \{ \phi = 0 \}$
be defined as the $0$ level set of a smooth, scalar function $\phi$
such that $\nabla \phi \neq 0$ on $\Gamma$. Take $k_u = 1$, $k_p = 2$
and let $\Gammah = \{\phi_h = 0 \}$ where $\phi_h \in X_h^1$ is the
Lagrange interpolant of $\phi$. Then $v_h = [\phi_h]^d \neq 0$ on $\mcT_h$ but
clearly $\| v_h\|_{\Gammah} = 0$.
Defining $q_h = \phi_h^2 \in \mcQ_h^2$ gives $\nabla q_h = 2 \phi_h \nabla \phi_h \neq 0$
but $\| \nabla q_h\|_{\Gammah} = 0$ in this particular case.
\end{remark}
Next, we state a simple surface-based discrete Poincar\'e estimate.
For a proof we refer to~\cite{BurmanHansboLarson2015}.
\begin{lemma}
  \label{lem:discrete-poincare-Gammah-stab}
  Let $v \in X_h$, then it holds
  \begin{align}
    \| v - \lambda_{\Gammah}(v)\|_{\Gamma} \lesssim \| \nablash v \|_{\Gammah}
    \label{eq:discrete-poincare-Gammah}
  \end{align}
  for $0 < h \leqslant h_0$ with $h_0$ chosen small enough.
\end{lemma}
Finally, the previous two lemma can be combined to
obtain the following discrete Poincar\'e inequality for the discrete ``energy'' norm $\tn V \tn_h$.
\begin{theorem}
  \label{thm:discrete-poincare-Nh-stab} For $(v,q) = V \in \VV_h$, the following estimate holds
  \begin{align}
    h^{-1}
    \left(
    \| v  \|^2_{\mcT_h}
    + \| q - \lambda_{\Gamma_h}(q) \|^2_{\mcT_h}
    \right)
    & \lesssim
      \tn V \tn_h
    \label{eq:discrete-poincare-Nh}
  \end{align}
  for $0<h \leq h_0$ with $h_0$ small enough.
\end{theorem}

\begin{figure}[htb!]
  \begin{subfigure}[t]{0.45\textwidth}
    \includegraphics[page=1,width=1\textwidth]{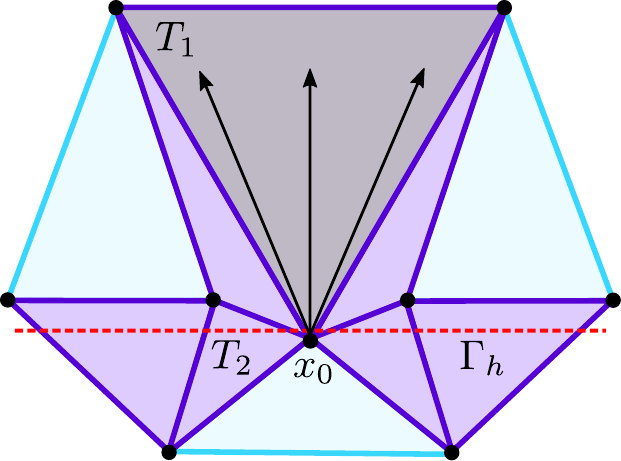}
  \end{subfigure}
  \hspace{0.04\textwidth}
    \begin{subfigure}[t]{0.45\textwidth}
    \includegraphics[page=1,width=1\textwidth]{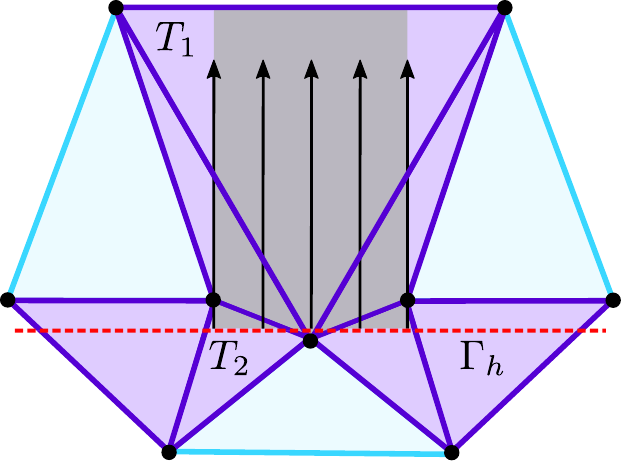}
    \end{subfigure}
    \caption{$L^2$ control mechanisms
for the full gradient and normal gradient stabilization.  (Left) While
element $T_1$ has only a small intersection with $\Gamma_h$, there are
several neighbor elements in $\mcT_h$ (purple) which share the node $x_0$ and
have a ``fat" intersection with $\Gamma_h$. The appearance of the full
gradient in stabilization $s_h^1$ allows to integrate along arbitrary
directions and thus gives raise to the control of $\| v \|_{T_1}$
through Lemma~\ref{lem:scaled-L2-norm-Nh-stab}.
(Right) The
fat intersection property for the discrete ``normal'' tube guarantees
that still a significant portion of $T_1$ can be reached when
integrating along normal-like paths which start from $\Gamma_h$ and
which reside completely inside $\mcT_h$.
    }
    \label{fig:condition_number-example}
\end{figure}

\subsection{Trace Estimates and Inverse Inequalities}
\label{ssec:trace-inverse-est}
First, we recall the following trace inequality for $v \in H^1(\mcT_h)$
\begin{equation}
  \| v \|_{\partial T} 
  \lesssim
  h^{-1/2} \|v \|_{T} +
  h^{1/2} \|\nabla v\|_{T}
  \quad \foralls T \in \mcT_h,
  \label{eq:trace-inequality}
\end{equation}
while for the intersection $\Gamma \cap T$ the corresponding inequality
\begin{align}
  \| v \|_{\Gamma \cap T} 
  \lesssim
  h^{-1/2} \| v \|_{T} 
  + h^{1/2} \|\nabla v\|_{T} 
  \quad \foralls T \in \mcT_h
  \label{eq:trace-inequality-cut-faces}
\end{align}
holds whenever $h$ is small enough,
see \cite{HansboHansboLarson2003} for a proof.
In the following, we will also need some
well-known inverse estimates for $v_h \in V_h$:
\begin{gather}
  \label{eq:inverse-estimate-grad}
  \| \nabla v_h\|_{T} 
  \lesssim
  h^{-1} 
  \| v_h\|_{T} \quad \foralls T \in \mcT_h,
  \\
  \| v_h\|_{\partial T} 
  \lesssim
  h^{-1/2} 
  \| v_h\|_{T},
  \qquad 
  \| \nabla v_h\|_{\partial T} 
  \lesssim
  h^{-1/2} 
  \| \nabla v_h\|_{T} \quad \foralls T \in \mcT_h,
  \label{eq:inverse-estimates-boundary}
\end{gather}
and the following ``cut versions'' 
when $K \cap T \not \subseteq \partial T$
\begin{alignat}{5}
  \|v_h \|_{K \cap T} 
  &\lesssim
  h^{-1/2} \|v_h\|_{T},
  & & \qquad 
  \| \nabla v_h \|_{K \cap T} 
  &\lesssim
  h^{-1/2} \|\nabla v_h\|_{T}
  & &\quad \foralls K \in \mcK_h, \;
  \foralls T \in \mcT_h,
  \label{eq:inverse-estimate-cut-v-on-K}
\end{alignat}
which are an immediate consequence of similar inverse estimates
presented in~\cite{HansboHansboLarson2003}.

\subsection{Geometric Estimates}
We now summarize some standard geometric identities and estimates
which typically are used in the numerical analysis of surface PDE
discretizations when passing from the discrete surface to the
continuous one and vice versa.  For a detailed derivation, we refer to
\cite{Dziuk1988,OlshanskiiReuskenGrande2009,Demlow2009,DziukElliott2013,BurmanHansboLarsonEtAl2016a}.
Starting with the Hessian of the signed distance function
\begin{align}
  \mcH = \nabla \otimes \nabla \rho \quad \text{on }
  U_{\delta_0}(\Gamma),
\end{align}
the derivative of the closest point projection 
and of an extended function $v^e$ is given by
\begin{gather}
Dp = \Ps (I - \rho \mcH) = \Ps - \rho \mcH,
\label{eq:derivative-closest-point-projection}
\\
  Dv^e = D(v \circ p) = Dv Dp = Dv P_{\Gamma}(I - \rho \mcH).
\label{eq:derivative-extended-function}
\end{gather}
The self-adjointness of $\Ps$, $\Psh$, and $\mcH$,
and the fact that $ \Ps \mcH = \mcH = \mcH \Ps$
and $\Ps^2 = \Ps$
then leads to the identities
\begin{align}
  \nabla v^e &= \Ps(I - \rho \mcH) \nabla v
  = \Ps(I - \rho \mcH) \nablas v,
  \label{eq:ve-full gradient}
  \\
  \nablash v^e &= \Psh(I - \rho \mcH)\Ps \nabla v = B^{T} \nablas v,
  \label{eq:ve-tangential-gradient}
\end{align}
where the invertible linear mapping
\begin{align}
  B = P_{\Gamma}(I - \rho \mcH) P_{\Gamma_h}: T_x(\Gammah) \to T_{p(x)}(\Gamma)
  \label{eq:B-def}
\end{align}
maps the tangential space of $\Gamma_h$ at $x$ to the tangential space of $\Gamma$ at
$p(x)$. Setting $v = w^l$ and using the identity $(w^l)^e = w$, we immediately get that
\begin{align}
  \nablas w^l = B^{-T} \nablash w
\end{align}
for any elementwise differentiable function $w$ on $\Gamma_h$ lifted to $\Gamma$.
We recall from \cite[Lemma 14.7]{GilbargTrudinger2001}
that for $x\in U_{\delta_0}(\Gamma)$, the Hessian $\mcH$
admits a representation
\begin{equation}\label{Hform}
  \mcH(x) = \sum_{i=1}^k \frac{\kappa_i^e}{1 + \rho(x)\kappa_i^e}a_i^e \otimes a_i^e,
\end{equation}
where $\kappa_i$ are the principal curvatures with corresponding
principal curvature vectors $a_i$.
Thus
\begin{equation}
  \|\mcH\|_{L^\infty(U_{\delta_0}(\Gamma))} \lesssim 1
  \label{eq:Hesse-bound}
\end{equation}
for $\delta_0 > 0$ small enough.
In the course of the a priori analysis in Section~\ref{sec:a-priori-est},
we will need to estimate various operator compositions involving
$B$, the continuous and discrete tangential and normal projection operators.
More precisely, 
using the definition
$\Qsh := I - \Psh = n_h \otimes n_h$,
the following bounds will be employed at several occasions.
\begin{lemma}
  \label{lem:composed-operator-bounds}
  \begin{alignat}{5}
    \| \Ps - \Ps \Psh \Ps \|_{L^\infty(\Gamma)} & \lesssim h^{2 k_g},
    \qquad & & 
    \| \Qsh \Ps\|_{L^\infty(\Gamma)} &\lesssim h^{k_g},
    \qquad & & 
    \|\Ps \Qsh \|_{L^\infty(\Gammah)} &\lesssim h^{k_g},
    \label{eq:projector-bounds}
    \\
    \| B \|_{L^\infty(\Gamma_h)} &\lesssim 1,
    \qquad & & 
    \| B^{-1} \|_{L^\infty(\Gamma)} & \lesssim 1,
    \qquad & & 
    \| P_\Gamma - B B^T \|_{L^\infty(\Gamma)} &\lesssim h^{k_g+1}.
    \label{eq:BBTbound}
  \end{alignat}
\end{lemma}
\begin{proof}
  All these estimate have been proved earlier, see
  \cite{Dziuk1988,DziukElliott2013,BurmanHansboLarsonEtAl2016c}
  and we only include a short proof for the reader's convenience.
  We start with the bounds summarized in~\eqref{eq:projector-bounds}.
  An easy calculation shows that
  $\Ps - \Ps \Psh \Ps = \Ps (\Ps - \Psh)^2 \Ps$ 
  from which the desired bound follows by observing that
  $
  \Ps - \Psh = 
  (n - n_h)  \otimes n
  +
  n_h \otimes (n - n_h)
  $ and thus 
  $ 
  \| (\Ps - \Psh)^2 \|_{L^{\infty}(\Gamma_h)} \lesssim
  \| n - n_h \|_{L^{\infty}(\Gamma_h)}^2 \lesssim h^{2 k_g}.
  $
  Next, observe that
  \begin{align}
  \| \Qsh \Ps \|_{L^{\infty}(\Gamma)}
  &= 
  \| n_h \otimes n_h - (n_h,n)_{\RR^d} n_h \otimes n \|_{L^{\infty}(\Gamma)}
  \\
  &= 
  \|(1 - (n_h,n)_{\RR^d}) n_h \otimes n_h \|_{L^{\infty}(\Gamma)} + \| (n_h,n)_{\RR^d} n_h
  \otimes (n_h - n) \|_{L^{\infty}(\Gamma)}
  \\
  &\lesssim h^{2 k_g} + h^{k_g}.
  \end{align}
  Turning to~\eqref{eq:BBTbound},
  the first two bounds follow directly from~\eqref{eq:B-def} and
  ~\eqref{eq:Hesse-bound} together with the assumption
  $\|\rho\|_{L^{\infty}(\Gamma_h)} \lesssim h^{k_g +1}$.
  Finally, unwinding the definition of $B$, we find that
  $
  \Ps - B B^T = \Ps - \Ps \Psh \Ps + O(h^{k_g+1}),
  $
  which together with the previously derived estimate for $\Ps - \Ps \Psh \Ps$
  gives the stated operator bound.
\end{proof}
The previous lemma allows us to
quantify the error introduced by using the
full gradient in~\eqref{eq:Ah-def} instead of $\nablash$.
To do so we decompose the full gradient as
$\nabla = \nablash + \Qsh \nabla$ 
with $\Qsh = I - \Psh = n_h \otimes n_h$.
We then have
\begin{corollary}
  \label{cor:normal-grad-est}
  For $v \in H^1(\Gamma)$  and $w \in V_h$ it holds
  \begin{align}
    \| \Qsh \nabla v^e \|_{\Gamma_h}
    \lesssim h^{k_g}
    \| \nablas v \|_{\Gamma},
\qquad
    \| \Ps \Qsh \nabla w \|_{\Gamma_h}
    \lesssim h^{k_g}
    \| \nabla w \|_{\Gamma_h}.
    \label{eq:normal-grad-est}
  \end{align}
\end{corollary}
\begin{proof}
  Since $\| \Qsh \Ps \|_{L^{\infty}(\Gamma_h)} \lesssim h^{k_g}$,
  the first estimate follows directly from the
  identity
  $\nabla v^e$ = $ \Ps(I - \rho \mcH)\nablas v$
  from~\eqref{eq:ve-full gradient}, while
    the second estimate is a immediate consequence of~\eqref{eq:projector-bounds}.
\end{proof}
Next, 
for a subset $\omega\subset \Gammah$,
we have the change of variables formula
\begin{equation}
\int_{\omega^l} g^l d\Gamma 
=  \int_{\omega} g |B|d\Gamma_h
\end{equation}
with $|B|$ denoting the absolute value of the determinant 
of $B$.  The determinant $|B|$ satisfies the following estimates.
\begin{lemma}  It holds
  \label{lem:detBbounds}
  \begin{alignat}{5}
  \| 1 - |B| \|_{L^\infty(\mcK_h)} 
  &\lesssim h^{k_g+1}, 
  & &\qquad
  \||B|\|_{L^\infty(\mcK_h)} 
  &\lesssim 1, 
  & &\qquad
  \||B|^{-1}\|_{L^\infty(\mcK_h)} 
  &\lesssim 1.
    \label{eq:detBbound}
\end{alignat}
\end{lemma}
\noindent Combining the various estimates for the norm and the determinant of $B$ shows
that for $m = 0,1$
\begin{alignat}{3}
  \| v \|_{H^{m}(\mcK_h^l)} &\sim \| v^e \|_{H^{m}(\mcK_h)}  
  & &\quad \text{for } v \in H^m(\mcK_h^l),
  \label{eq:norm-equivalences-ve}
  \\
  \| w^l \|_{H^{m}(\mcK_h^l)} &\sim \| w \|_{H^{m}(\mcK_h)}
  & &\quad \text{for } w \in V_h.
  \label{eq:norm-equivalences-wh}
\end{alignat}

Next, we observe that thanks to the coarea-formula (cf. \citet{EvansGariepy1992})
\begin{align*}
\int_{U_{\delta}} f(x) \,dx = \int_{-\delta}^{\delta} 
\left(\int_{\Gamma(r)} f(y,r) \, \mathrm{d} \Gamma_r(y)\right)\dr,
\end{align*}
the extension operator $v^e$ defines a bounded operator
$H^m(\Gamma) \ni v \mapsto v^e \in H^m(U_{\delta}(\Gamma))$
satisfying the stability estimate
\begin{align}
  \| v^e \|_{k,U_{\delta}(\Gamma)} \lesssim \delta^{1/2} \| v
  \|_{k,\Gamma}, \qquad 0 \leqslant k \leqslant m
\label{eq:stability-estimate-for-extension}
\end{align}
for $0 < \delta \leqslant \delta_0$, where the hidden constant depends only on the curvature of $\Gamma$.

\subsection{Interpolation Operator}
Next, we recall from~\cite{ErnGuermond2004} that for
$v \in H^{k+1}(N_h)$,
the Cl\'ement interpolant $\pi_h:L^2(\mcT_h) \rightarrow X_h^k$ satisfies the
local interpolation estimates
\begin{alignat}{3}
\| v - \pi_h v \|_{m,T} 
& \lesssim
  h^{k + 1 -m}| v |_{k+1,\omega(T)},
  & &\quad 0\leqslant m \leqslant k+1, \quad &\foralls T\in \mcT_h,
  \label{eq:interpest0}
\end{alignat}
where $\omega(T)$ consists of all elements sharing a
vertex with $T$. 
Now with the help of the extension operator $(\cdot)^e$,
an interpolation operator $\pi_h: L^2(\Gamma) \to X^k_h$ 
can be constructed by
setting $\pi_h v = \pi_h v^e$, where we took the liberty of using the same symbol.
The resulting interpolation operator satisfies the following error estimate.

\begin{lemma}
\label{lem:interpolenergy}
For $V = (v,p) \in [H^{k_u}(\Gamma)]^d \times H^{k_p+1}(\Gamma)$ and $k_u, k_p \geqslant 1$,
the interpolant defined by $\Pi_h V^e = (\pi_h v^e,
\pi_h q^e) \in \mcV_h^{k_u} \times \mcQ_h^{k_p}$ satisfies
the interpolation estimate
\begin{align}
  \label{eq:interpolenergy}
  \tn V^e - \Pi_h V^e \tn_h
  \lesssim
  h^{k_u} \| v \|_{k_u, \Gamma}
  + h^{k_p} \| q \|_{k_p+1, \Gamma}.
\end{align}
\end{lemma}
\begin{proof}
Choosing $\delta_0 \sim h$,
it follows directly from combining the trace inequality~\eqref{eq:trace-inequality-cut-faces}, the interpolation
estimate~\eqref{eq:interpest0}, and the stability
estimate~\eqref{eq:stability-estimate-for-extension} that
the first two terms in the definition of 
  $
  \tn V^e - \Pi_h V^e \tn_h^2
  = \| v^e - \pi_h v^e \|_{\Gamma}^2
  + \| \nabla(p^e - \pi_h p^e) \|_{\Gamma}^2
  + | V^e - \Pi_h V^e |_{S_h}^2
  $
  satisfies the desired estimate.
  Since $|\cdot|_{S_h^2} \leqslant |\cdot|_{S_h^1}$ it is enough
  to focus on the full gradient stabilization $S_h = S_h^1$
  for the remaining part. With the same chain of estimates we find that
  \begin{align}
    | V^e - \Pi_h V^e |_{S_h}^2
    &= h \bigl(
      \| \nabla(v^e - \pi_h v^e) \|_{\mcT_h}^2
      +
  \| \nabla(p^e - \pi_h p^e) \|_{\mcT_h}^2
      \bigr)
    \\
    &
    \lesssim
    h^{2k_u-1}
    \| v^e \|_{k_u, \mcT_h}^2
    +
    h^{2k_p+1}
      \| q^e \|_{k_p+1, \mcT_h}^2
    \\
    &
    \lesssim
    h^{2k_u}
    \| v \|_{k_u, \Gamma}^2
    +
    h^{2k_p+2}
    \| q \|_{k_p+1, \Gamma}^2
  \end{align}
  which concludes the proof.
\end{proof}

\section{A Priori Error Estimates}
\label{sec:a-priori-est}
We now state and prove the main a priori error estimates for the stabilized
cut finite element method~(\ref{eq:darcy-probl-cutfem}).
The proofs rest upon a Strang-type lemma splitting the total error
into an interpolation error, a consistency error arising from
the additional stabilization term $S_h$ and finally, 
a geometric error caused by the discretization of the surface.
We start with establishing suitable estimates for the consistency and quadrature error
before we present the final a priori error estimates
at the end of this section.

\subsection{Estimates for the Quadrature and Consistency Error}
\label{ssec:quadr-error-estim}
The purpose of the next lemma is two-fold. First, it shows that
the full gradient stabilization will not affect the expected convergence order
when low-order elements are used. Second, it demonstrates that
only the normal gradient stabilization is suitable for high order discretizations
where the geometric approximation order $k_g$ needs to satisfy $k_g > 1$.
\begin{lemma}
  Let  $U = (u,p) \in [H^1(\Gamma)]^d_t \times  H^1(\Gamma)$. Then it holds
  \label{lem:Sh-const-est}
  \begin{align}
    |U^e|_{S_h^1} &\lesssim h (
                    \| \nablas u \|_{\Gamma}
                    +
                    \| \nablas p \|_{\Gamma}
                    ),
                  \\
    |U^e|_{S_h^2} & \lesssim
                    h^{k_g + 1}
                    (
                    \| \nablas u \|_{\Gamma}
                    +
                    \| \nablas p \|_{\Gamma}
                    ).
    \label{eq:Sh1-const-est}
  \end{align}
\end{lemma}
\begin{proof}
  A simple application of stability estimate~\eqref{eq:stability-estimate-for-extension}
  with $\delta \sim h$ shows that for $S_h^1$,
  \begin{align}
    S_h^1(U^e, U^e) = 
    h \|\nabla u^e\|_{\mcT_h}^2
  + h \|\nabla p^e\|_{\mcT_h}^2
  \lesssim
  h^2 (\| u \|_{1,\Gamma}^2
  +
  \| p \|_{1,\Gamma}^2).
  \end{align}
  Turning to $S_h^2$,  the pressure part of the normal gradient stabilization
  can be estimated by
\begin{align}
  s_p(p^e, p^e)
  =
  h
  \| \Qsh \nabla p^e \|_{\mcT_h}^2 =
  h
  \| (\Qsh - \Qs) \nabla p^e \|_{\mcT_h}^2
  \lesssim
  h^{2k_g + 1} \| \nabla p^e \|_{\mcT_h}^2
  \lesssim
  h^{2 k_g + 2} \| \nablas p \|_{\Gamma}^2,
\end{align}
and similarly, $|u^e|_{s_h^2} \lesssim  h^{k_g + 1} \| \nablas u \|_{\Gamma}$
for $u \in [H^1(\Gamma)]^d$.
\end{proof}
\begin{lemma}
  Let $U = (u,p) \in [L^2(\Gamma)]_t^d \times H^1(\Gamma)/\RR $
  be the solution to weak problem~(\ref{eq:darcy-problem-weak-tang}) and assume that
   $V \in \VV_h$.
  Then
  \begin{align}
    | L(V^l) - L_h(V)|
    + |A(U, V^l) - A_h(U^e, V_h)|
    &\lesssim h^{k_g} (\|f\|_{\Gamma} + \|g\|_{\Gamma}) \tn V \tn_h.
    \label{eq:Lh-quad-est-primal}
  \end{align}
  Furthermore, if
  $\Phi = (\phi_u, \phi_p) \in [H^1(\Gamma)]^d_t \times H^2(\Gamma)/\RR$
  and $\Phi_h := \Pi_h \Phi = (\pi_h \phi_u, \pi_h \phi_p)$, we
  have the improved estimate
  \begin{align}
    | L(\Phi_h^l) - L_h(\Phi_h)|
    + 
    |A(U, \Phi_h^l) - A_h(U^e, \Phi_h)|
    &\lesssim h^{k_g+1} (\|f\|_{\Gamma} + \|g\|_{\Gamma})
     ( \| \phi_u \|_{1,\Gamma} + \| \phi_p \|_{2,\Gamma}).
    \label{eq:Lh-quad-est-dual}
  \end{align}
\end{lemma}
\begin{proof}
  We start with the term $L(\cdot) - L_h(\cdot)$. Unwinding the definition of the
  linear forms $L$ and $L_h$, we get
  \begin{align}
    L(V^l) - L_h(V)
    &=
      \Bigl(
      (f, q^l)_{\Gamma} - 
      (f^e, q)_{\Gamma_h}
      + \onehalf
      \left(
      (g, v^l)_{\Gamma}
      - (g^e, v)_{\Gamma_h}
      \right)
      \Bigr)
      \\
      &\quad + \onehalf
      \left(
      (g, \nabla q^l)_{\Gamma}
      - (g^e, \nabla q)_{\Gamma_h}
        \right)
        = I + II.
  \end{align}
  For the first term, a change of variables together with estimate~\eqref{eq:detBbound}
  for the determinant $|B|$ yields
  \begin{align}
    I &=  (f, (1 - |B|^{-1}) q^l)_{\Gamma} + \onehalf (g, (1 - |B|^{-1}) v^l)_{\Gamma}
        \\
    &\lesssim h^{k_g +1} \left(
      \| f \|_{\Gamma}
      + \|g\|_{\Gamma}
      \right)
      (\| q^l \|_{\Gamma} + \| v^l \|_{\Gamma})
      \\
    &\lesssim h^{k_g + 1} \left(
      \| f \|_{\Gamma}
      + \|g\|_{\Gamma}
      \right)
      \tn V \tn_h,
  \end{align}
  where in the last step, we used the norm equivalences~(\ref{eq:norm-equivalences-wh})
  and the discrete Poincar\'e inequality~(\ref{eq:discrete-poincare-Gammah})
  to pass to $\tn V_h \tn_h$.
  To estimate $II$, we split $\nabla q$ into its tangential and normal part
  \begin{align}
    \nabla q = \nablash q + \Qsh \nabla q.
  \end{align}
  Note that for the tangential field $g$,
  the identities 
  \begin{align}
    (g, \nabla q^l)_{\Gamma} = (g, \nablas q^l)_{\Gamma},
    \quad
    (g^e, \nabla q)_{\Gamma_h} = (\Ps g^e, \nabla q)_{\Gammah}
    \label{eq:tangential-g-identities}
  \end{align}
  hold and thus using  $\Ps g = g$ once more and the fact that $\Ps^T  = \Ps$ allows us to rewrite
  $II$ as
  \begin{align}
    2 II &= (g, \nablas q^l)_{\Gamma} - (g^e, \nablash q)_{\Gamma_h} - (g^e, \Qsh \nabla q)_{\Gamma_h}
           \\
    &= (g, (\Ps - |B|^{-1}B^T) \nablas q^l)_{\Gamma} - (\Ps g^e,\Qsh \nabla q)_{\Gammah}
         \\
       &= (g, \Ps (\Ps - |B|^{-1}B^T) \nablas q^l)_{\Gamma} + (g^e,\Ps\Qsh \nabla q)_{\Gammah}
         \\
       &= II_t + II_n.
  \end{align}
  Unwinding the definition of $B$ given in~(\ref{eq:B-def}) together with the
  estimates for the determinant $|B|$ from Lemma~\ref{lem:detBbounds}
  reveals that
  \begin{align}
    \Ps (\Ps - |B|^{-1}B^T)
    &=
    \Ps (\Ps - B^T) +  \Ps(|B|^{-1} - 1)B^T
   \\
    &\sim \Ps (\Ps - \Psh(I - \rho \mcH)\Ps) + h^{k_g + 1}
   \\&\sim \Ps - \Ps\Psh\Ps + h^{k_g + 1}.
  \end{align}
  Consequently, 
  using the bounds for $\Ps - \Ps \Psh \Ps$
  and $\Ps\Qsh$ from Lemma~\ref{lem:composed-operator-bounds}, we deduce that
  \begin{align}
    II_t & \lesssim h^{k_g + 1} \| g \|_{\Gamma} \| \nablas q^l \|_{\Gamma}
           \lesssim h^{k_g + 1} \| g \|_{\Gamma} \tn V \tn_h,
           \\
    II_n & \lesssim h^{k_g} \| g \|_{\Gamma} \| \nabla q \|_{\Gammah}
           \lesssim h^{k_g} \| g \|_{\Gamma} \tn V \tn_h.
  \end{align}
  In the special case $q = \pi_h \phi_p^e$, the bound for $II_n$ can be further improved to
  \begin{align}
    II_n &= (g^e,\Ps\Qsh \Ps \nabla \phi_p^e)_{\Gammah}
           + (g^e,\Ps\Qsh \nabla ( \pi_h \phi_p^e - \phi_p^e))_{\Gammah},
           \label{eq:Lhdual_IIn-1}
    \\
 &\lesssim
   h^{k_g + 1} \| g \|_{\Gamma} \| \nabla \phi_p \|_{\Gamma}
   + 
   h^{k_g + 1} \| g \|_{\Gamma} \| \phi_p \|_{2,\Gamma},
  \label{eq:Lhdual_IIn-2}
  \end{align}
  where we once more employed the identity $\nabla \phi_p^e = \Ps \nabla\phi_p^e$,
  the estimates~(\ref{eq:projector-bounds}) for the operators $\Ps \Qsh \Ps $ and $\Ps\Qsh$ and finally,
  the interpolation estimate~(\ref{eq:interpolenergy}).
  
  Turning to the term $A(U, \cdot) - A(U^e, \cdot)$
  in~(\ref{eq:Lh-quad-est-primal}) and~(\ref{eq:Lh-quad-est-dual}) and
  recalling the definition of bilinear forms $A$ and $A_h$, we rearrange terms
  to obtain
  \begin{align}
    2 \left(
    A(U, V^l) - A_h(U^e, V)
    \right)
    &=
      \left(
    (u, v^l)_{\Gamma} - (u^e, v)_{\Gamma_h}
    \right)
      +
      \left(
    (\nabla p, v^l)_{\Gamma}
      - (\nabla p^e, v)_{\Gammah}
      \right)
    \\
    &\quad
      -
      \left(
    (u, \nabla q^l)_{\Gamma}
      - (u^e, \nabla q)_{\Gammah}
     \right) 
      +
      \left(
    (\nabla p, \nabla q^l)_{\Gamma}
      -(\nabla p^e, \nabla q)_{\Gamma_h}
      \right)
    \\
    &= I + II - III + IV.
  \end{align}
  To estimate the term $I$--$IV$,
  we proceed along the same lines as in the previous part. As before, 
  the first term can be bounded as follows
  \begin{align}
    I &= (u, (1 - |B|^{-1}) v^l)_{\Gamma}
    \lesssim h^{k_g + 1} \| u \|_{\Gamma} \| v^l \|_{\Gamma}
    \lesssim h^{k_g + 1} \| u \|_{\Gamma} \tn V \tn_h.
  \end{align}
  For the remaining terms,
  the appearance of the full gradient 
  necessitates a similar split into its normal and tangential part
  as before, followed by a lifting of the tangential part
  and the use of the operator estimates~(\ref{eq:projector-bounds}) and~(\ref{eq:BBTbound}).
  Recall that $\nabla p = \nablas p^e$ and consequently,
  \begin{align}
    II &= (\nablas p, v^l)_{\Gamma} - (\nablash p^e, v)_{\Gammah} - (\Qsh \nabla p^e, v)_{\Gammah}
         \\
       &=  ((\Ps - |B|^{-1}B^T) \nablas p, v^l)_{\Gamma} - (\Qsh \Ps \nabla p^e, v)_{\Gammah}
    \\
    &= II_t + II_n.
  \end{align}
  Now expand $B$ to see that
  $\Ps - |B|^{-1}B^T \sim \Ps - \Psh \Ps + h^{k_g+1}  \sim \Qsh \Ps + h^{k_g + 1}$
  and apply the
  operator bounds from Lemma~\ref{lem:composed-operator-bounds}
  to $\Qsh \Ps$,
  followed by the norm equivalences~(\ref{eq:norm-equivalences-wh})
  to arrive at the following estimates
  \begin{align}
    II_t &\lesssim |(\Qsh \Ps \nablas p, v^l)_{\Gamma}| + h^{k_g + 1} \| \nablas p \|_{\Gamma} \tn V \tn_h
           \lesssim (h^{k_g} + h^{k_g + 1}) \| \nablas p \|_{\Gamma} \tn V \tn_h,
           \label{eq:Ah-dual-IIt}
           \\
    II_n & \lesssim h^{k_g} \| \nablas p \|_{\Gamma} \tn V \tn_h.
  \end{align}
  In the special case $v = \pi_h \phi_u$, exploiting that $\phi_u$ is
  a $H^1$ regular, tangential field and applying the proper operator and
  interpolation estimates,
  the bounds for $II_n$ can be improved to
  \begin{align}
    II_n &= (\Qsh \Ps \nablas p^e, \pi_h \phi_u^e)_{\Gammah}
    \\
    &= ( \Ps \Qsh \Ps \nablas p^e, \phi_u^e)_{\Gammah} +
    (\Qsh \Ps \nablas p^e, \pi_h \phi_u^e - \phi_u^e)_{\Gammah}
    \\
         &\lesssim h^{k_g + 1} \| \nablas p \|_{\Gamma} \|\phi_u \|_{\Gamma}
           + h^{k_g} \| \nablas p \|_{\Gamma} \|\pi_h \phi_u^e - \phi_u^e\|_{\Gamma}
           \lesssim
         h^{k_g + 1} \| \nablas p \|_{\Gamma} \|\phi_u \|_{1,\Gamma},
  \end{align}
  and similarly for $II_t$, the improvement of
  first term in~\eqref{eq:Ah-dual-IIt} gives 
  \begin{align}
    II_t &\lesssim h^{k_g + 1} \| \nablas p \|_{\Gamma} \|\phi_u \|_{1,\Gamma}.
  \end{align}
  Turning to the third term, we rewrite $III$ as
  \begin{align}
    III &=  (\Ps u, \nablas q^l)_{\Gamma} - (\Ps u^e, \nablash q)_{\Gammah}
          -  (\Ps u_e, \Qsh \nabla q)_{\Gammah}
          \\
    &= (\Ps u, (\Ps - |B|^{-1}B^T) \nablas q^e)_{\Gammah}
      - (u^e, \Ps \Qsh \nabla q)_{\Gammah} = III_t + III_n.
  \end{align}
  Using $\Ps (\Ps - |B|^{-1}B^T) \sim \Ps - \Ps  \Psh \Ps + h^{k_g+1}$ and
  applying the operator bounds~(\ref{eq:projector-bounds}) yields
  \begin{align}
    III_t &\lesssim  h^{k_g+1} \| u \|_{\Gamma} \| \nablas q^l \|_{\Gamma}
            \lesssim h^{k_g+1} \| u \|_{\Gamma} \tn V \tn_h,
            \\
    III_n &\lesssim  h^{k_g} \| u^e \|_{\Gammah} \| \nabla q \|_{\Gammah}
            \lesssim h^{k_g} \| u \|_{\Gamma} \tn V \tn_h.
  \end{align}
  Following precisely steps~(\ref{eq:Lhdual_IIn-1})--(\ref{eq:Lhdual_IIn-2}),
  the term $III_n$ can be improved if $q = \pi_h \phi_p$, showing that
  \begin{align}
   III_t &\lesssim
           h^{k_g + 1} \| u \|_{\Gamma} \| \phi_p \|_{2,\Gamma}.
  \end{align}
  Finally, starting from the fact that $\nabla p = \nablas p$,
  similar steps lead the following bound for $IV$
  \begin{align}
    IV &= (\nablas p, \nablas q^l)_{\Gamma}
         - (\nablash p^e, \nablash q)_{\Gammah}
         - (\Qsh \nabla p^e, \Qsh \nabla q)_{\Gammah}
         \\
    &= ((\Ps - |B|^{-1}B B^T) \nablas p, \nabla q^l)_{\Gamma}
      - (\Qsh \nabla p^e, \Qsh \nabla q)_{\Gammah}
      = IV_t + IV_n,
      \\
    \intertext{and as before thanks to~(\ref{eq:BBTbound}), (\ref{eq:normal-grad-est})
    and interpolation estimate~(\ref{eq:interpolenergy}), we see that
    }
    IV_t &\lesssim h^{k_g +1} \|\nablas p \|_{\Gamma} \| \nablas q \|_{\Gammah}
           \lesssim h^{k_g +1} \|\nablas p \|_{\Gamma} \tn V \tn_h,
           \\
    IV_n &\lesssim h^{k_g} \|\nablas p \|_{\Gamma} \| \nabla q \|_{\Gammah}
           \lesssim h^{k_g} \|\nablas p \|_{\Gamma} \tn V \tn_h,
           \\
    IV_n &\lesssim h^{k_g + 1} \| \nablas p \|_{\Gamma} \| \phi_p \|_{2,\Gamma},
  \end{align}
  assuming $q = \pi_h \phi_p$ in the last case.
  Collecting the estimates for $I$--$IV$ and using the stability estimate
  $\tn U \tn \lesssim (\|f\|_{\Gamma} + \|g\|_{\Gamma})$
  concludes the proof.
\end{proof}

\subsection{A Priori Error Estimates}
\label{ssec:priori-error-estim}
We start with establishing an a priori estimate for the error
measured in the natural ``energy'' norm.
\begin{theorem}
  \label{thm:priori-error-estim}
  Let $U = (u,p)$ be the solution to the continuous problem~(\ref{eq:darcy-problem-cont}).
  Assume that $(u,p) \in [H^{k_u+1}(\Gamma)]^d_t \times H^{k_p+1}(\Gamma)$ and that the geometric
  assumptions~(\ref{eq:geometric-assumptions-II}) hold.
  Then for the full gradient stabilized form $B_h = A_h + S_h^1$, the
  solution $U_h = (u_h, p_h) \in \mcV_h^k \times \mcQ_h^l$
  to the discrete problem~(\ref{eq:darcy-probl-cutfem})
  satisfies  the a priori estimate
  \begin{align}
    \tn U^e - U_h \tn_h
    \lesssim
    h^{k_u +1} \| u\|_{k+1,\Gamma}
    + h^{k_p} \| p\|_{l+1,\Gamma}
    + h^{k_g} (\| f \|_{\Gamma} + \| g\|_{\Gamma})
    + h (\|u\|_{1,\Gamma} + \|p\|_{1,\Gamma}).
    \label{eq:aprior-est-energy-sh1}
  \end{align}
  If the normal gradient stabilization $S_h =  S_h^2$ is employed instead,
  the discretization error satisfies the improved estimate 
  \begin{align}
    \tn U^e - U_h \tn_h
    \lesssim
    h^{k_u +1} \| u\|_{k+1,\Gamma}
    + h^{k_p} \| p\|_{l+1,\Gamma}
    + h^{k_g} (\| f \|_{\Gamma} + \| g\|_{\Gamma})
    + h^{k_g+1} (\|u\|_{1,\Gamma} + \|p\|_{1,\Gamma}).
    \label{eq:aprior-est-energy-sh2}
  \end{align}
\end{theorem}
\begin{proof}
  We start with considering the ``discrete error'' $E_h = U_h - V_h$. Observe that
  \begin{align}
    \tn E_h \tn_h^2
    & = B_h(U_h - V_h,E_h)
    \\
    & =
    L_h(E_h) - B_h(U^e, E_h) + B_h(U^e - V_h,E_h)
    \\
    & \lesssim
    \biggl(
    \sup_{V_h \in \VV_h}
    \dfrac{
      L_h(V_h) - B_h(U^e,V_h)
    }{
      \tn V_h \tn_h
    }
    +
      \tn U^e - V_h \tn_h
    \biggr)
    \tn E_h \tn_h.
    \label{eq:strang-energy-lemma-step-3}
  \end{align}
  Dividing by $\tn E_h \tn_h$ and applying the identity
  \begin{align}
    L_h(V_h) - B_h(U^e, V_h)
    &=
    \bigl(
    L_h(V_h) -
    L(V_h^l)
    \bigr)
    +
    \bigl(A(U,V_h^l) - A_h(U^e, V_h)
    \bigr)- S_h(U^e, V_h)
  \end{align}
  gives together with the triangle inequality
  $\tn U^e - U_h \tn_h \leqslant \tn U^e - V_h \tn_h
  + \tn E_h \tn_h$ 
  the following Strang-type estimate for the energy error,
  \begin{align}
    \tn U^e - U_h \tn_h
    &\lesssim
    \inf_{V_h \in \VV_h}
    \tn  U^e - V_h \tn_h
    + \sup_{V_h \in \VV_h}
    \dfrac{
      L_h(V_h) - B_h(U^e,V_h)
    }{
      \tn V_h \tn_h
    }
    \label{eq:strang-energy-1}
    \\
    &\lesssim
    \inf_{V_h \in \VV_h}
    \tn U^e - V_h \tn_h
    + \sup_{V_h \in \VV_h}
    \dfrac{
      L_h(V_h) - L(V_h^l)
    }{
      \tn V_h \tn_h
    }
    + \sup_{V_h \in \VV_h}
    \dfrac{
      A(U,V_h^l) - A_h(U^e,V_h)
    }{
      \tn V_h \tn_h
    }
    \nonumber
    \\
    &\phantom{\leqslant}
    + \sup_{v \in \VV_h}
    \dfrac{
      S_h(U^e,V_h)
    }{
      \tn V_h \tn_h
    }.
      \label{eq:strang-energy-2}
  \end{align}
  Estimates~(\ref{eq:aprior-est-energy-sh1}) and~(\ref{eq:aprior-est-energy-sh2})
  now follow directly from
  inserting the interpolation estimate~(\ref{eq:interpolenergy}),
  the quadrature error estimate~\eqref{eq:Lh-quad-est-primal}
  and, depending on the choice
  of $S_h$, the proper consistency error estimate from Lemma~\ref{lem:Sh-const-est}
  into~\eqref{eq:strang-energy-2}.
\end{proof}
Next, we  provide bounds for the $L^2$ error of the pressure approximation as well
as the $H^{-1}$ error of the tangential component of the
velocity approximation.
\begin{theorem}
  \label{thm:priori-error-estim-dual}
  Under the same assumptions as made in Theorem~\ref{thm:priori-error-estim},
  the following a priori error estimate holds
  \begin{align}
    \| p - p_h^l \|_{\Gamma}
    + \| \Ps (u - u_h^l) \|_{-1,\Gamma}
    &\lesssim
      h C_U,
  \end{align}
  with $C_U$ being the convergence rate predicted by Theorem~\ref{thm:priori-error-estim}.
\end{theorem}
\begin{proof}
  The proof uses a standard Aubin-Nitsche duality argument 
  employing the dual problem
  \begin{subequations}
    \label{eq:dual-problem}
    \begin{alignat}{2}
      -\divs \phi_u &= \psi_p  &\quad \text{on } \Gamma,
      \\
      \phi_u - \nablas \phi_p &= \psi_u  &\quad \text{on $\Gamma$},
    \end{alignat}
  \end{subequations}
  with $(\psi_u, \psi_p) \in [H^1(\Gamma)]^d_t \times L^2_0(\Gamma)$.
  For the error representation to be derived it is sufficient to
  consider $(\psi_u, \psi_p)$ such that $ \| \psi_u \|_{1,\Gamma} +  \| \psi_p \|_{\Gamma} \lesssim 1$.
  Thanks to the regularity result~(\ref{eq:ellreg}), the
solution $(\phi_u, \phi_p) \in [H^1(\Gamma)]^d_t \times H^2(\Gamma)
\cap L^2_0(\Gamma)$ then satisfies the stability estimate
  \begin{align}
    \| \phi_u \|_{1,\Gamma} + \| \phi_p \|_{2,\Gamma}
    \lesssim 1.
    \label{eq:stability-dual-problem}
  \end{align}
  Set $E = U - U_h^l$ and
  insert the dual solution $\Phi$ as test function into $A(E,\cdot$).
  Then adding and subtracting suitable terms leads us to
  \begin{align}
    A(E,\Phi) &= A(E,\Phi - \Phi_h^l) + A(E, \Phi_h^l)
      \label{eq:ep-error-repres-I}
    \\
    &= A(E,\Phi - \Phi_h^l) + L(\Phi_h^l) - A(U_h^l, \Phi_h^l)
    \\
    &= A(E,\Phi - \Phi_h^l) +
      \left(
      L(\Phi_h^l) - L_h(\Phi_h)
      \right)+
      \left(
      A_h(U_h, \Phi_h) - A(U_h^l, \Phi_h^l)
      \right)
      \\
      &\quad + S_h(U_h, \Phi_h)
    \\
    &= I + II + III + IV.
      \label{eq:ep-error-repres-final}
  \end{align}
  where in the last step, we employed the identity $B_h(U_h, \Phi_h) - L_h(\Phi_h) = 0$.
  Interpolation estimate~(\ref{eq:interpolenergy}) together with stability
  estimate~(\ref{eq:stability-dual-problem}) implies that
  \begin{align}
    I &\lesssim \tn E \tn  \tn \Phi - \Phi_h^l \tn
        \lesssim h \tn E \tn_h
        (
        \| \phi_u \|_{1,\Gamma}
        +
        \| \phi_p \|_{2,\Gamma}
        )
        \lesssim h \tn E \tn_h.
  \end{align}
  Next, the improved quadrature error estimates~(\ref{eq:Lh-quad-est-dual}) 
  and the stability bound~(\ref{eq:stability-dual-problem}) imply that
  \begin{align}
    II + III & \lesssim
    h^{k_g+1}
    (\|f\|_{\Gamma} + \|g\|_{\Gamma})
    (\| \phi_u \|_{1,\Gamma} + \|\phi_p\|_{2,\Gamma})
    \lesssim
    h^{k_g+1}
    (\|f\|_{\Gamma} + \|g\|_{\Gamma}).
  \end{align}
  Finally, after adding and subtracting $U^e$ and $\Phi^e$,
  the consistency error can be bounded
  by
  \begin{align}
    IV &=
    S_h(U_h - U^e, \Phi_h - \Phi^e)
    + S_h(U_h - U^e, \Phi^e)
    + S_h(U^e, \Phi_h - \Phi^e)
         + S_h(U^e, \Phi^e)
    \\
    &\lesssim
    \tn U_h - U^e \tn_h \tn \Phi_h - \Phi^e \tn_h
    + \tn U_h - U^e \tn_h  | \Phi^e|_{S_h}
    + | U^e |_{S_h}  \tn \Phi_h - \Phi^e \tn_h
    +
    | \Phi^e|_{S_h}
      | U^e |_{S_h}
    \\
    &\lesssim h C_U,
  \end{align}
  where in the last step,
  the energy error estimate from Theorem~\ref{thm:priori-error-estim},
  the interpolation estimate~(\ref{eq:interpolenergy}), the consistency error estimates from
  Lemma~\ref{lem:Sh-const-est} and the stability bound~(\ref{eq:stability-dual-problem})
  were successively applied.
  Collecting the estimates for term $I$--$IV$ shows that 
  \begin{align}
    |A(E,\Phi)| \lesssim  h C_U.
    \label{eq:Ah-with-Phi-bound}
  \end{align}
  
  Next,  using the shorthand notation $E = (e_u, e_p) = (u-u_h^l, p-p_h^l)$,
  we exploit the properties of the dual problem to derive an error representation
  for $\|e_p\|_{\Gamma}$ and $\|\Ps e_u \|_{-1,\Gamma}$
  in terms of $A(E, \Phi)$ to establish the desired bounds using~\eqref{eq:Ah-with-Phi-bound}.
  Since \mbox{$\lambda_{\Gamma_h}(p_h) = 0$} but not necessarily $\lambda_{\Gamma}(p_h^l)$,
  we first decompose the pressure error $e_p$ into a normalized part $\widetilde{e}_p$
  satisfying $\lambda_{\Gamma}(\widetilde{e}_p) = 0$ and a constant part $\overline{e}_p$,
  \begin{align}
    e_p = p - p_h^l =
    \underbrace{p - (p_h^l - \lambda_{\Gamma}(p_h^ l))}_{\widetilde{e}_p}
    +
    \underbrace{\lambda_{\Gamma}(p_h^ l) - \lambda_{\Gammah}(p_h) }_{\overline{e}_p}.
  \end{align}
  Then the properties of dual solution $\Phi$ 
  together with the observations that $\phi_u = \Ps\phi_u$, $\nabla \phi_p = \nablas \phi_p$
  and $\nabla e_p = \nabla \widetilde{e}_p$  lead us to the identity
  \begin{align}
    A(E,\Phi)
    &= (e_u, \phi_u)_{\Gamma}
    - (e_u, \nabla \phi_p)_{\Gamma}
    + (\nabla \widetilde{e}_p, \phi_u)_{\Gamma}
    + \onehalf (e_u + \nabla e_p, -\phi_u + \nabla \phi_p)_{\Gamma}
    \\
    &= (e_u, \psi_u)_{\Gamma} - (\widetilde{e}_p, \divs \phi_u)_{\Gamma}
      - \onehalf (e_u + \nabla e_p, \psi_u)_{\Gamma}
      \\
    &=
      \onehalf (e_u, \psi_u)_{\Gamma}
      + (\widetilde{e}_p, \psi_p)_{\Gamma}
      + \onehalf (e_p, \divs \psi_u)_{\Gamma}.
      \label{eq:dual-error-repres}
  \end{align}
  Thus choosing $\psi_u = 0$ and $\psi_p \in L^2_0(\Gamma)$,
  the normalized pressure error can be bounded as follows
  \begin{align}
    \| \widetilde{e}_p \|_{\Gamma}
    &= \sup_{\psi \in L^2_0(\Gamma), \|\psi_p\|_{\Gamma} = 1} (e_p, \psi_p)_{\Gamma}
    = \sup_{\psi \in L^2_0(\Gamma), \|\psi_p\|_{\Gamma} = 1} A(E,\Phi(0,\psi_p))
    \lesssim h C_U.
  \end{align}
  Turning to constant error part $\overline{e}_p$ and
  unwinding the definition of the average operators
  $\lambda_{\Gammah}(\cdot)$ and $\lambda_{\Gamma}(\cdot)$
  yields
  \begin{align}
    \|\overline{e}_p\|_{\Gamma} 
    = |\Gamma|^{\onehalf}
    \left|
    \dfrac{1}{|\Gamma|}
    \int_{\Gamma} p_h^l  \ds
   - 
    \dfrac{1}{|\Gamma_h|}
    \int_{\Gamma_h} p_h  \dsh
   \right|
   \lesssim
    \dfrac{|\Gamma_h|^{\onehalf}}{|\Gamma_h|}
    \int_{\Gamma_h} |1-c| |p_h|  \dsh,
\label{eq:error-of-average-est}
  \end{align}
  with $c = |\Gamma_h||\Gamma|^{-1} |B|$.
  We note  that 
  $\| 1 - c \|_{L^\infty(\Gamma)} \lesssim h^{k_g+1}$
  thanks to~\eqref{eq:detBbound}. Consequently, after
  successively applying a Cauchy-Schwarz inequality,
  the Poincar\'e inequality~(\ref{eq:discrete-poincare-Gammah})
  and the stability bound
  $\| \nablash p_h \|_{\Gammah} \lesssim \|f \|_{\Gamma} + \|g \|_{\Gamma}$,
  we arrive at 
  \begin{align}
    \|\overline{e}_p\|_{\Gamma} 
    \lesssim
    h^{k_g+1} \| p_h\|_\Gamma
    \lesssim
    h^{k_g+1} \| \nabla p_h\|_{\Gammah}
    \lesssim
    h^{k_g+1}
    (\|f \|_{\Gamma} + (\|g \|_{\Gamma}),
  \end{align}
  which concludes the derivation of the desired estimate for $\| e_p \|_{\Gamma}$.

  Finally, to estimate $\|\Ps(u-u_h^l)\|_{-1,\Gamma}$, we let $\Phi$ be the solution
  to the dual problem~(\ref{eq:dual-problem}) for
  right-hand side data $(\psi_u, 0)$ with $\psi_u \in [H^1(\Gamma)]^d_t$. Inserting $\Phi$
  into~\eqref{eq:dual-error-repres} 
  \begin{align}
    |(e_u, \psi_u)_{\Gamma}| \lesssim |A(E, \Phi)| + \|e_p\|_{\Gamma} \|\psi_u\|_{1,\Gamma},
  \end{align}
  and consequently, the general bound~(\ref{eq:Ah-with-Phi-bound}) for $A(E,\Phi)$
  together with bound for $L^2$ error of the pressure allows
  us to derive the final estimate for $e_u$,
  \begin{align}
    \| \Ps e_u \|_{-1, \Gamma}
    &=
    \sup_{\psi_u \in [H^1(\Gamma)]^d_t, \|\psi_u\|_{-1,\Gamma} = 1}
      (e_u, \psi_u)_{\Gamma}
      \lesssim
      |A(E, \Phi)| + \|e_p\|_{\Gamma}
      \lesssim h C_u.
  \end{align}
\end{proof}

\section{Numerical Results}
\label{sec:numerical-results}
To numerically examine the rate of convergence predicted by the a priori error
estimates derived in Section~\ref{ssec:priori-error-estim},
we now perform a series of convergence studies.
Following the numerical example presented in~\cite{HansboLarson2016},
we consider the Darcy problem
posed on the torus surface $\Gamma$ defined by
\begin{align}
  \Gamma
  =
  \{
    x \in \RR^3 : r^2 = x_3^2 + (\sqrt{x_1^2 + x_2^2} -R)^2
  \},
\label{eq:torus-levelset}
\end{align}
with major radius $R = 1.0$ and minor radius $r = 0.5$ and
define a manufactured solution $(u,p)$ by
\begin{align}
  u_t =
  \Bigl(
  2xz, -2yz, 2(x^2 - y^2)(R- \sqrt{x^2 + y^2})/\sqrt{x^2 + y^2}
  \Bigr),
  \quad
  u_n = 0,
  \quad
  p = z,
\end{align}
which satisfies the Darcy problem~(\ref{eq:darcy-problem-cont})
with right-hand sides $f = \divs u = 0$ and
\begin{align}
  g =
  \begin{pmatrix}
    xz(2 - (1 - \tfrac{R}{\sqrt{x^2 + y^2}})/A)
    \\
    yz(-2 - (1 - \tfrac{R}{\sqrt{x^2 + y^2}})/A)
    \\
  1 - \tfrac{2(x^2 - y^2)(\sqrt{x^2 + y^2} -R)}{\sqrt{x^2 + y^2}} - z^2/A
\end{pmatrix},
  \quad
  \text{with }
  A = (R^2 + x^2 + y^2 - 2R\sqrt{x^2 + y^2} + z^2).
\end{align}
A sequence of meshes $\{\mcT_k\}_{k=0}^l$ with uniform
mesh sizes $h_k = 2^{-k} h_0$ with $h_0 \approx 0.24$ is generated by uniformly
refining an initial, structured background mesh $\widetilde{\mcT}_0$
for $\Omega = [-1.65,1.65]^3 \supset \Gamma$
and extracting at each refinement level $k$
the active (background) mesh
as defined by \eqref{eq:narrow-band-mesh}.
\begin{figure}[htb]
    \begin{subfigure}[b]{0.60\textwidth}
    \includegraphics[width=1.0\textwidth]{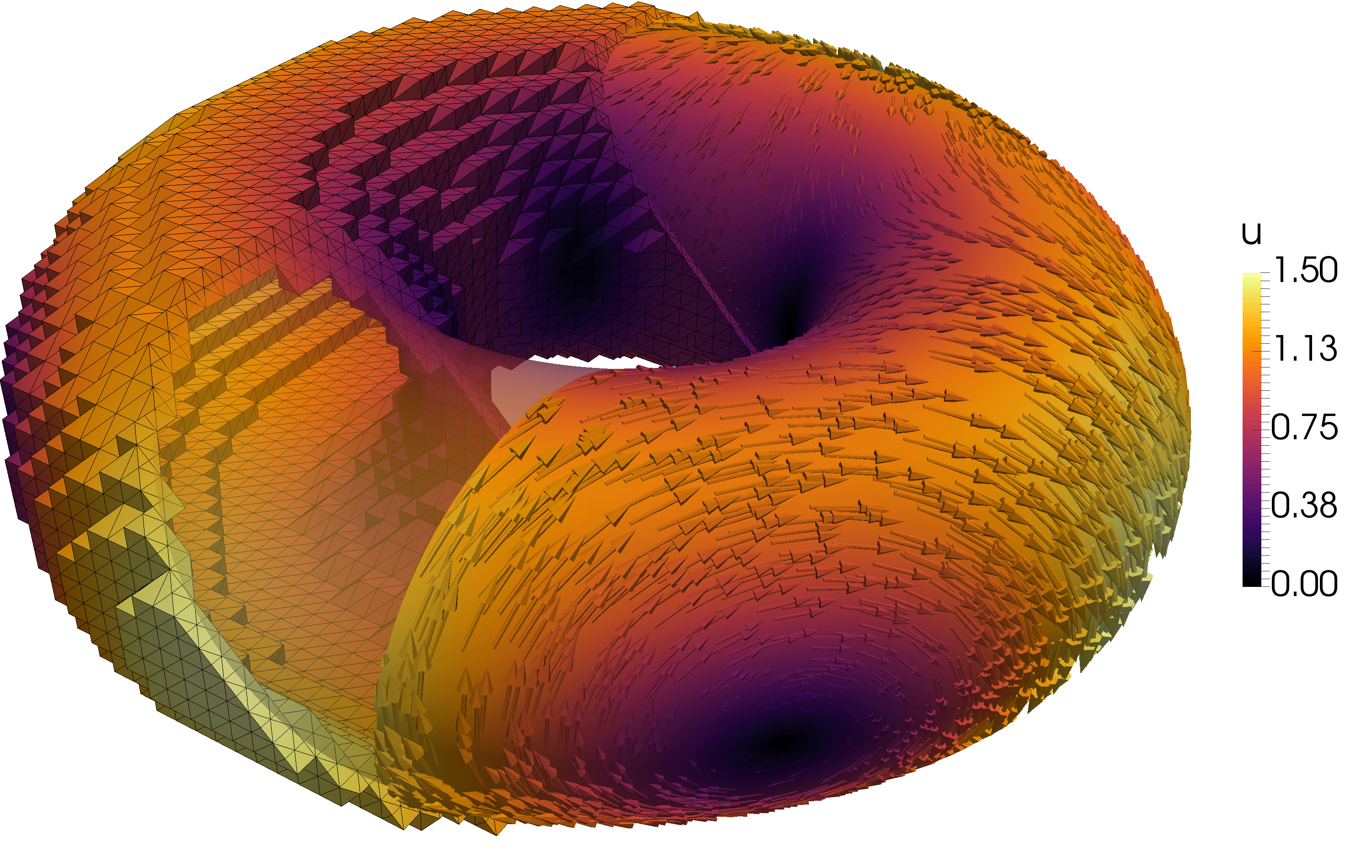}
    \end{subfigure}
    \\[2ex]
    \begin{subfigure}[b]{0.60\textwidth}
    \includegraphics[width=1.0\textwidth]{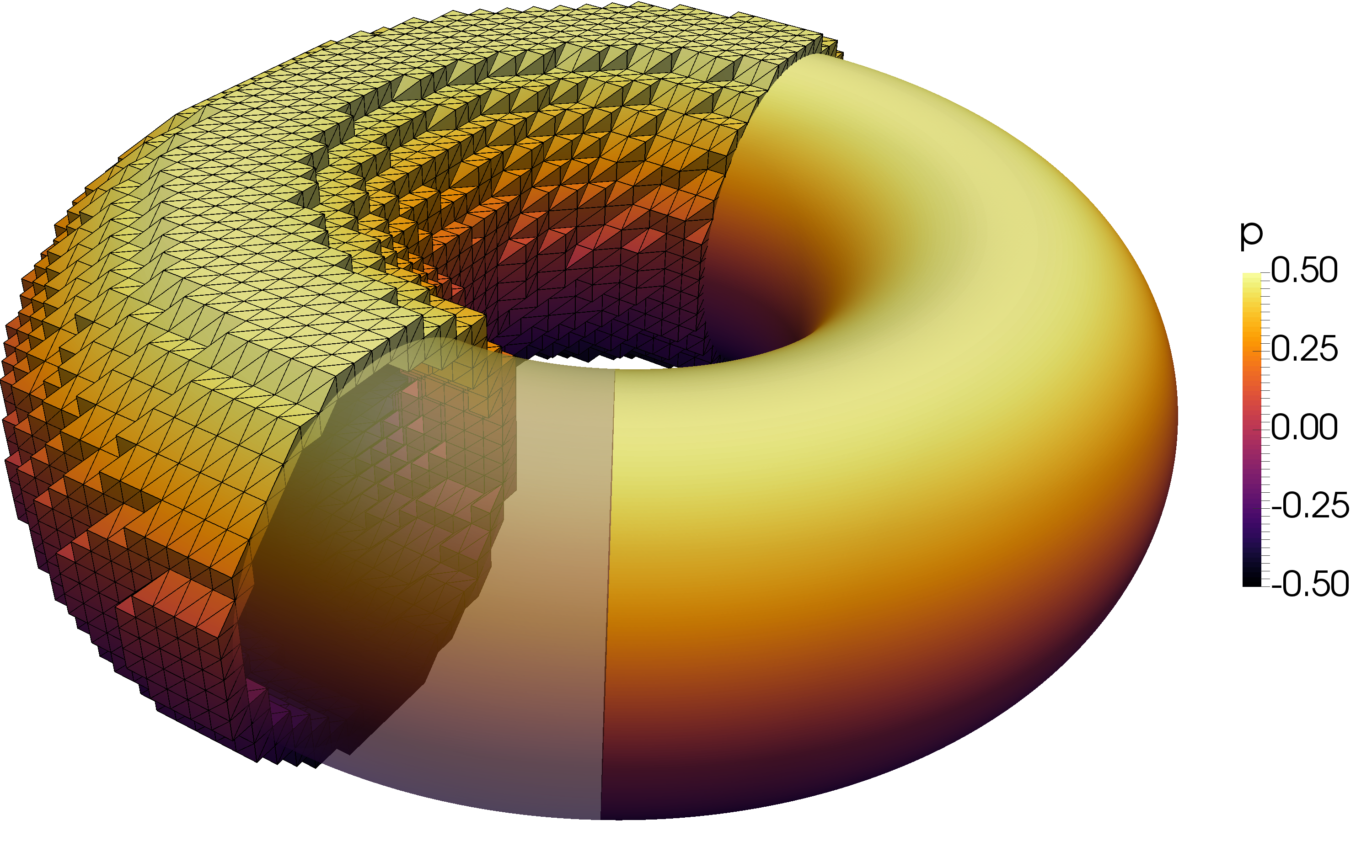}
    \end{subfigure}
    \hspace{0.5ex}
    \caption{Plots of the velocity (top) and pressure (bottom) approximations
      computed for $(k_u, k_p, k_g) = (1,2,2)$ on the finest refinement level.
      Each plot shows both the solution as
      computed on the active mesh $\mcT_h$ and its restriction 
      to the surface mesh $\mcK_h$.
      For the velocity, the magnitude and the computed vector field
      are displayed, illustrating the weak enforcement of the tangential
      condition $u \cdot n = 0$ in the discrete vector field $u_h$.
    }
    \label{fig:solution-plots}
\end{figure}
For a given error norm, the corresponding
experimental order of convergence (EOC) at
refinement level $k$ is calculated using the formula
\begin{align*}
  \text{EOC}(k) = \dfrac{\log(E_{k-1}/E_{k})}{\log(2)},
\end{align*}
with $E_k$ denoting error of the computed discrete velocity
$u_k$ or pressure  $p_k$ at refinement level $k$.

To study the combined effect of chosing various approximation orders
$k_u$, $k_p$ and $k_g$ and stabilization forms $S_h^i$ on the overall
approximation quality of the discrete solution,
we conduct convergence experiments for 6 different scenarios.
For each scenario, we compute the $L^2$ norm of the velocity error
$e_u = u^e - u_h$ as well as the $H^1$ and $L^2$ norms of the
pressure error $e_p = p^e - p_h$ which are displayed in Table~\ref{tab:eoc-tables}.
A short summary of the cases considered and the theoretically expected convergence rates
are given Table~\ref{tab:eoc-cases}.
  The computed EOC data in Table~\ref{tab:eoc-tables} clearly confirms the
  predicted convergence rates.
  In particular, we observe that increasing the pressure approximation to $k_p = 2$
  does only increase the convergence order for all considered error norms by one
  if both a high order approximation $k_g = 2$ of $\Gammah$ and
  the higher-order consistent normal stabilization $S_h^2$ are used.
  Finally, the discrete solution components computed for $(k_u, k_p, k_g) = (1,2,2)$ at
  the finest refinement level are visualized in Figure~\ref{fig:solution-plots}.
  \begin{table}[htb]
  \centering
   \footnotesize
  \begin{tabular}{c c c c c c c c}
    \toprule
    Case & $k_u$ & $k_p$ & $k_g$ & $S_h$ & $\| e_u \|_{\Gammah}$ & $\| e_u \|_{1,\Gammah}$ & $\| e_p \|_{\Gammah}$
    \\
    \midrule
    1 & 1 & 1 & 1 & $S_h^1$ & 1 & 1 & 2 \\
    2 & 1 & 1 & 1 & $S_h^2$ & 1 & 1 & 2 \\
    3 & 1 & 2 & 1 & $S_h^1$ & 1 & 1 & 2 \\
    4 & 1 & 2 & 1 & $S_h^2$ & 1 & 1 & 2 \\
    5 & 1 & 2 & 2 & $S_h^1$ & 1 & 1 & 2 \\
    6 & 1 & 2 & 2 & $S_h^2$ & 2 & 2 & 3 \\
    \bottomrule
  \end{tabular}
  \caption{Summary of the 6 cases considered in the convergence
experiments and the corresponding theoretical convergence rates
predicted by Theorem~\ref{thm:priori-error-estim}
and~Theorem~\ref{thm:priori-error-estim-dual}.
  }
    \label{tab:eoc-cases}
   \footnotesize
    \end{table}
\begin{table}[htb]
  \centering
  \begin{subtable}[t]{1.0\textwidth}
    \centering
    \begin {tabular}{cr<{\pgfplotstableresetcolortbloverhangright }@{}l<{\pgfplotstableresetcolortbloverhangleft }cr<{\pgfplotstableresetcolortbloverhangright }@{}l<{\pgfplotstableresetcolortbloverhangleft }cr<{\pgfplotstableresetcolortbloverhangright }@{}l<{\pgfplotstableresetcolortbloverhangleft }c}%
\toprule $k$&\multicolumn {2}{c}{$\|u_k - u^e \|_{\Gamma _h}$}&EOC&\multicolumn {2}{c}{$\|p_k - p^e \|_{1,\Gamma _h}$}&EOC&\multicolumn {2}{c}{$\|p_k - p^e \|_{\Gamma _h}$}&EOC\\\midrule %
\pgfutilensuremath {0}&$9.71$&$\cdot 10^{-1}$&--&$1.10$&$\cdot 10^{0}$&--&$1.69$&$\cdot 10^{-1}$&--\\%
\pgfutilensuremath {1}&$3.06$&$\cdot 10^{-1}$&\pgfutilensuremath {1.67}&$6.46$&$\cdot 10^{-1}$&\pgfutilensuremath {0.77}&$5.29$&$\cdot 10^{-2}$&\pgfutilensuremath {1.68}\\%
\pgfutilensuremath {2}&$9.04$&$\cdot 10^{-2}$&\pgfutilensuremath {1.76}&$3.10$&$\cdot 10^{-1}$&\pgfutilensuremath {1.06}&$1.18$&$\cdot 10^{-2}$&\pgfutilensuremath {2.16}\\%
\pgfutilensuremath {3}&$3.08$&$\cdot 10^{-2}$&\pgfutilensuremath {1.55}&$1.53$&$\cdot 10^{-1}$&\pgfutilensuremath {1.02}&$2.80$&$\cdot 10^{-3}$&\pgfutilensuremath {2.08}\\%
\pgfutilensuremath {4}&$1.30$&$\cdot 10^{-2}$&\pgfutilensuremath {1.25}&$7.72$&$\cdot 10^{-2}$&\pgfutilensuremath {0.99}&$6.86$&$\cdot 10^{-4}$&\pgfutilensuremath {2.03}\\\bottomrule %
\end {tabular}%

    \\[1ex]
    \caption{Case 1: $(k_u, k_p, k_g) = (1,1,1)$ and full gradient stabilization $S_h = S_h^1$.}
  \end{subtable}
  \begin{subtable}[t]{1.0\textwidth}
    \centering
    \begin {tabular}{cr<{\pgfplotstableresetcolortbloverhangright }@{}l<{\pgfplotstableresetcolortbloverhangleft }cr<{\pgfplotstableresetcolortbloverhangright }@{}l<{\pgfplotstableresetcolortbloverhangleft }cr<{\pgfplotstableresetcolortbloverhangright }@{}l<{\pgfplotstableresetcolortbloverhangleft }c}%
\toprule $k$&\multicolumn {2}{c}{$\|u_k - u^e \|_{\Gamma _h}$}&EOC&\multicolumn {2}{c}{$\|p_k - p^e \|_{1,\Gamma _h}$}&EOC&\multicolumn {2}{c}{$\|p_k - p^e \|_{\Gamma _h}$}&EOC\\\midrule %
\pgfutilensuremath {0}&$4.70$&$\cdot 10^{-1}$&--&$1.28$&$\cdot 10^{0}$&--&$7.31$&$\cdot 10^{-2}$&--\\%
\pgfutilensuremath {1}&$1.43$&$\cdot 10^{-1}$&\pgfutilensuremath {1.72}&$6.45$&$\cdot 10^{-1}$&\pgfutilensuremath {0.99}&$2.08$&$\cdot 10^{-2}$&\pgfutilensuremath {1.82}\\%
\pgfutilensuremath {2}&$4.75$&$\cdot 10^{-2}$&\pgfutilensuremath {1.58}&$3.14$&$\cdot 10^{-1}$&\pgfutilensuremath {1.04}&$4.84$&$\cdot 10^{-3}$&\pgfutilensuremath {2.10}\\%
\pgfutilensuremath {3}&$2.06$&$\cdot 10^{-2}$&\pgfutilensuremath {1.21}&$1.57$&$\cdot 10^{-1}$&\pgfutilensuremath {0.99}&$1.19$&$\cdot 10^{-3}$&\pgfutilensuremath {2.02}\\%
\pgfutilensuremath {4}&$9.92$&$\cdot 10^{-3}$&\pgfutilensuremath {1.05}&$7.80$&$\cdot 10^{-2}$&\pgfutilensuremath {1.01}&$2.82$&$\cdot 10^{-4}$&\pgfutilensuremath {2.08}\\\bottomrule %
\end {tabular}%

    \\[1ex]
    \caption{Case 2: $(k_u, k_p, k_g) = (1,1,1)$ and normal gradient stabilization $S_h = S_h^2$.}
  \end{subtable}
  \begin{subtable}[t]{1.0\textwidth}
    \centering
    \begin {tabular}{cr<{\pgfplotstableresetcolortbloverhangright }@{}l<{\pgfplotstableresetcolortbloverhangleft }cr<{\pgfplotstableresetcolortbloverhangright }@{}l<{\pgfplotstableresetcolortbloverhangleft }cr<{\pgfplotstableresetcolortbloverhangright }@{}l<{\pgfplotstableresetcolortbloverhangleft }c}%
\toprule $k$&\multicolumn {2}{c}{$\|u_k - u^e \|_{\Gamma _h}$}&EOC&\multicolumn {2}{c}{$\|p_k - p^e \|_{1,\Gamma _h}$}&EOC&\multicolumn {2}{c}{$\|p_k - p^e \|_{\Gamma _h}$}&EOC\\\midrule %
\pgfutilensuremath {0}&$4.69$&$\cdot 10^{-1}$&--&$1.30$&$\cdot 10^{0}$&--&$7.42$&$\cdot 10^{-2}$&--\\%
\pgfutilensuremath {1}&$1.43$&$\cdot 10^{-1}$&\pgfutilensuremath {1.72}&$6.46$&$\cdot 10^{-1}$&\pgfutilensuremath {1.01}&$2.08$&$\cdot 10^{-2}$&\pgfutilensuremath {1.83}\\%
\pgfutilensuremath {2}&$4.75$&$\cdot 10^{-2}$&\pgfutilensuremath {1.58}&$3.14$&$\cdot 10^{-1}$&\pgfutilensuremath {1.04}&$4.85$&$\cdot 10^{-3}$&\pgfutilensuremath {2.10}\\%
\pgfutilensuremath {3}&$2.06$&$\cdot 10^{-2}$&\pgfutilensuremath {1.21}&$1.57$&$\cdot 10^{-1}$&\pgfutilensuremath {0.99}&$1.19$&$\cdot 10^{-3}$&\pgfutilensuremath {2.02}\\%
\pgfutilensuremath {4}&$9.92$&$\cdot 10^{-3}$&\pgfutilensuremath {1.05}&$7.80$&$\cdot 10^{-2}$&\pgfutilensuremath {1.01}&$2.82$&$\cdot 10^{-4}$&\pgfutilensuremath {2.08}\\\bottomrule %
\end {tabular}%

    \\[1ex]
    \caption{Case 3: $(k_u, k_p, k_g) = (1,2,1)$ and full gradient stabilization $S_h = S_h^1$.}
  \end{subtable}
  \begin{subtable}[t]{1.0\textwidth}
    \centering
    \begin {tabular}{cr<{\pgfplotstableresetcolortbloverhangright }@{}l<{\pgfplotstableresetcolortbloverhangleft }cr<{\pgfplotstableresetcolortbloverhangright }@{}l<{\pgfplotstableresetcolortbloverhangleft }cr<{\pgfplotstableresetcolortbloverhangright }@{}l<{\pgfplotstableresetcolortbloverhangleft }c}%
\toprule $k$&\multicolumn {2}{c}{$\|u_k - u^e \|_{\Gamma _h}$}&EOC&\multicolumn {2}{c}{$\|p_k - p^e \|_{1,\Gamma _h}$}&EOC&\multicolumn {2}{c}{$\|p_k - p^e \|_{\Gamma _h}$}&EOC\\\midrule %
\pgfutilensuremath {0}&$4.69$&$\cdot 10^{-1}$&--&$1.30$&$\cdot 10^{0}$&--&$7.42$&$\cdot 10^{-2}$&--\\%
\pgfutilensuremath {1}&$1.43$&$\cdot 10^{-1}$&\pgfutilensuremath {1.72}&$6.46$&$\cdot 10^{-1}$&\pgfutilensuremath {1.01}&$2.08$&$\cdot 10^{-2}$&\pgfutilensuremath {1.83}\\%
\pgfutilensuremath {2}&$4.75$&$\cdot 10^{-2}$&\pgfutilensuremath {1.58}&$3.14$&$\cdot 10^{-1}$&\pgfutilensuremath {1.04}&$4.85$&$\cdot 10^{-3}$&\pgfutilensuremath {2.10}\\%
\pgfutilensuremath {3}&$2.06$&$\cdot 10^{-2}$&\pgfutilensuremath {1.21}&$1.57$&$\cdot 10^{-1}$&\pgfutilensuremath {0.99}&$1.19$&$\cdot 10^{-3}$&\pgfutilensuremath {2.02}\\\bottomrule %
\end {tabular}%

    \\[1ex]
    \caption{Case 4: $(k_u, k_p, k_g) = (1,2,1)$ and normal gradient stabilization $S_h = S_h^2$.}
  \end{subtable}
  \begin{subtable}[t]{1.0\textwidth}
    \centering
    \begin {tabular}{cr<{\pgfplotstableresetcolortbloverhangright }@{}l<{\pgfplotstableresetcolortbloverhangleft }cr<{\pgfplotstableresetcolortbloverhangright }@{}l<{\pgfplotstableresetcolortbloverhangleft }cr<{\pgfplotstableresetcolortbloverhangright }@{}l<{\pgfplotstableresetcolortbloverhangleft }c}%
\toprule $k$&\multicolumn {2}{c}{$\|u_k - u^e \|_{\Gamma _h}$}&EOC&\multicolumn {2}{c}{$\|p_k - p^e \|_{1,\Gamma _h}$}&EOC&\multicolumn {2}{c}{$\|p_k - p^e \|_{\Gamma _h}$}&EOC\\\midrule %
\pgfutilensuremath {0}&$3.61$&$\cdot 10^{0}$&--&$1.56$&$\cdot 10^{0}$&--&$3.58$&$\cdot 10^{-1}$&--\\%
\pgfutilensuremath {1}&$2.47$&$\cdot 10^{0}$&\pgfutilensuremath {0.55}&$5.03$&$\cdot 10^{-1}$&\pgfutilensuremath {1.63}&$1.44$&$\cdot 10^{-1}$&\pgfutilensuremath {1.31}\\%
\pgfutilensuremath {2}&$9.04$&$\cdot 10^{-1}$&\pgfutilensuremath {1.45}&$1.34$&$\cdot 10^{-1}$&\pgfutilensuremath {1.91}&$3.40$&$\cdot 10^{-2}$&\pgfutilensuremath {2.09}\\%
\pgfutilensuremath {3}&$2.65$&$\cdot 10^{-1}$&\pgfutilensuremath {1.77}&$4.09$&$\cdot 10^{-2}$&\pgfutilensuremath {1.71}&$8.94$&$\cdot 10^{-3}$&\pgfutilensuremath {1.93}\\\bottomrule %
\end {tabular}%

    \\[1ex]
    \caption{Case 5: $(k_u, k_p, k_g) = (1,2,2)$ and full gradient stabilization $S_h = S_h^1$.}
  \end{subtable}
  \begin{subtable}[t]{1.0\textwidth}
    \centering
    \begin {tabular}{cr<{\pgfplotstableresetcolortbloverhangright }@{}l<{\pgfplotstableresetcolortbloverhangleft }cr<{\pgfplotstableresetcolortbloverhangright }@{}l<{\pgfplotstableresetcolortbloverhangleft }cr<{\pgfplotstableresetcolortbloverhangright }@{}l<{\pgfplotstableresetcolortbloverhangleft }c}%
\toprule $k$&\multicolumn {2}{c}{$\|u_k - u^e \|_{\Gamma _h}$}&EOC&\multicolumn {2}{c}{$\|p_k - p^e \|_{1,\Gamma _h}$}&EOC&\multicolumn {2}{c}{$\|p_k - p^e \|_{\Gamma _h}$}&EOC\\\midrule %
\pgfutilensuremath {0}&$1.55$&$\cdot 10^{0}$&--&$1.77$&$\cdot 10^{0}$&--&$1.71$&$\cdot 10^{-1}$&--\\%
\pgfutilensuremath {1}&$2.94$&$\cdot 10^{-1}$&\pgfutilensuremath {2.40}&$4.12$&$\cdot 10^{-1}$&\pgfutilensuremath {2.10}&$1.72$&$\cdot 10^{-2}$&\pgfutilensuremath {3.31}\\%
\pgfutilensuremath {2}&$4.73$&$\cdot 10^{-2}$&\pgfutilensuremath {2.63}&$9.63$&$\cdot 10^{-2}$&\pgfutilensuremath {2.10}&$1.49$&$\cdot 10^{-3}$&\pgfutilensuremath {3.53}\\%
\pgfutilensuremath {3}&$8.64$&$\cdot 10^{-3}$&\pgfutilensuremath {2.45}&$2.33$&$\cdot 10^{-2}$&\pgfutilensuremath {2.05}&$1.15$&$\cdot 10^{-4}$&\pgfutilensuremath {3.69}\\\bottomrule %
\end {tabular}%

    \\[1ex]
    \caption{Case 6: $(k_u, k_p, k_g) = (1,2,2)$ and normal gradient stabilization $S_h = S_h^2$.}
  \end{subtable}
  \caption{Experimental order of convergence for the all 6 cases
  computed with a stabilization parameter $\tau = 0.1$.}
  \label{tab:eoc-tables}
\end{table}
\clearpage

\section*{Acknowledgements}
This research was supported in part by
the Swedish Foundation for Strategic Research Grant
No.\ AM13-0029, the Swedish Research Council Grants Nos.\ 2011-4992,
2013-4708, and Swedish strategic research programme eSSENCE.

\bibliographystyle{plainnat}
\bibliography{bibliography}

\end{document}